\documentclass[11pt,a4paper]{article}

\usepackage{pgf,tikz,tkz-graph}
\usepackage{physics}
\usepackage{listings}
\usepackage{csquotes}
\usepackage[english]{babel}

\usepackage{mathrsfs}
\usetikzlibrary{arrows}
\usetikzlibrary{arrows.meta}
\usepackage[T1]{fontenc}
\usepackage{amsfonts}
\usepackage{amssymb}
\usepackage{amsthm}
\usepackage{amscd}
\usepackage{graphicx}
\usepackage{enumitem}
\usepackage{authblk}
\usepackage{verbatim}
\usepackage{hyperref}
\usepackage{amsmath,caption}
\usepackage{url,pdfpages,xcolor,framed,color}
\usepackage{a4wide}

\theoremstyle{plain}

\newtheorem{thr}{Theorem}[section]

\newtheorem{lem}[thr]{Lemma}
\newtheorem{prop}[thr]{Proposition}
\newtheorem{conj}[thr]{Conjecture}
\theoremstyle{definition}
\newtheorem{defi}[thr]{Definition}

\def\P{\mathcal{P}}
\def\Q{\mathcal{Q}}

\def\vc{\overrightarrow}

\DeclareMathOperator{\ecc}{ecc}
\DeclareMathOperator{\rad}{rad}

\title{Extremal total distance of graphs of given radius I}
\author{Stijn Cambie\footnote{Department of Mathematics, Radboud University Nijmegen, Postbus 9010, 6500 GL Nijmegen, The Netherlands. Email: \href{mailto:stijn.cambie@hotmail.com}{stijn.cambie@hotmail.com}. This work has been supported by a Vidi Grant of the Netherlands Organization for Scientific Research (NWO), grant number $639.032.614$.} }%
\date{}

\begin{document}
	\definecolor{xdxdff}{rgb}{0.49019607843137253,0.49019607843137253,1.}
	\definecolor{ududff}{rgb}{0.30196078431372547,0.30196078431372547,1.}
	
	\tikzstyle{every node}=[circle, draw, fill=black!50,
	inner sep=0pt, minimum width=4pt]

	\maketitle

	\begin{abstract}
		In 1984, Plesn\'{i}k determined the minimum total distance for given order and diameter and characterized the extremal graphs and digraphs.
		We prove the analog for given order and radius, when the order is sufficiently large compared to the radius.
		This confirms asymptotically a conjecture of Chen et al.
		We also state an analog of the conjecture of Chen et al for digraphs and prove it for sufficiently large order.
	\end{abstract}

	\section{Introduction}

The total distance $W(G)$ of a graph $G$ equals the sum of distances between all unordered pairs of vertices, i.e. $W(G)=\sum_{\{u,v\} \subset V} d(u,v).$
In 1984, Plesn\'{i}k~\cite{P84} determined the minimum total distance among all graphs of order $n$ and diameter $d$. He did this both for graphs and digraphs and characterized the extremal examples.
In this paper we solve the analogous questions for given order $n$ and radius $r$, when $n$ is sufficiently large compared with $r$.

The extremal graphs attaining the minimal total distance among all graphs with order $n$ and radius $r \in \{1,2\}$ are easily characterized; complete graphs when $r=1$, complete graphs minus a maximum matching when $r=2$ and $2 \mid n$ and complete graphs minus a maximum matching and an additional edge adjacent to the vertex not in the maximum matching, when $r=2$ and $2 \nmid n.$
For $r\ge 3$ the question is harder and a conjecture of the extremal graphs was made by Chen, Wu and An~\cite{Chen}. 
Here $G_{n,r,s}$ is a cycle $C_{2r}$ in which we take blow-ups in $2$ consecutive vertices by cliques $K_s$ and $K_{n-2r+2-s}$, as defined in Section~\ref{not&def}.
\begin{conj}[\cite{Chen}]\label{conjchen}
	Let $n$ and $r$ be two positive integers with $n \ge 2r$ and $r \ge 3$. For any graph $G$ of order $n$ with radius $r$, $W(G) \ge W(G_{n,r,1})$. Equality holds if and only if $G \cong G_{n,r,s}$ where $1 \le s \le \frac{n-2r+2}{2}$.
\end{conj}

Although there are counterexamples to Conjecture~\ref{conjchen} when $n$ is small (as we discuss below), we will show that Conjecture~\ref{conjchen} is true asymptotically, i.e. when $n\ge n_1(r)$ for some value $n_1(r)$, in Section~\ref{AsProofChen}.
\begin{thr}\label{main}
	For any $r \ge 3$, there exists a value $n_1(r)$ such that for all $n \ge n_1(r)$ the following hold
	\begin{itemize}
		\item any graph $G$ of order $n$ with radius $r$ satisfies $W(G) \ge W(G_{n,r,1})$. Equality holds if and only if $G \cong G_{n,r,s}$ where $1 \le s \le \frac{n-2r+2}{2}$.
	\end{itemize}
	
\end{thr}

There are some main ideas in the proof which are somewhat intuitive.
Since the minimum average distance is close to $1$, we expect there are many vertices of high degree and there is a large clique.
Because the conjectured extremal graphs contain large blow-ups, we can expect there are vertices such that $G \backslash v$ satisfies the original statement as well. By proving that the total distance $W(G)$ and $W(G \backslash v)$ differ by a certain amount with equality if and only if a structure close to the conjectured structure appears, at the end we only need to prove that an extremal graph is exactly of the form $G_{n,r,s}.$

For small values of $n$ with respect to a fixed $r$, there might be a few exceptions to Conjecture~\ref{conjchen}. 
The graph $Q_3$ is a counterexample for the equality statement when $r=3$ and $n=8$, as it also has a total distance equal to $48.$ 
A computer check~\footnote{See \url{https://github.com/StijnCambie/ChenWuAn}, document SmallN\_CWA.} has shown that this is the only counterexample for $n<10.$

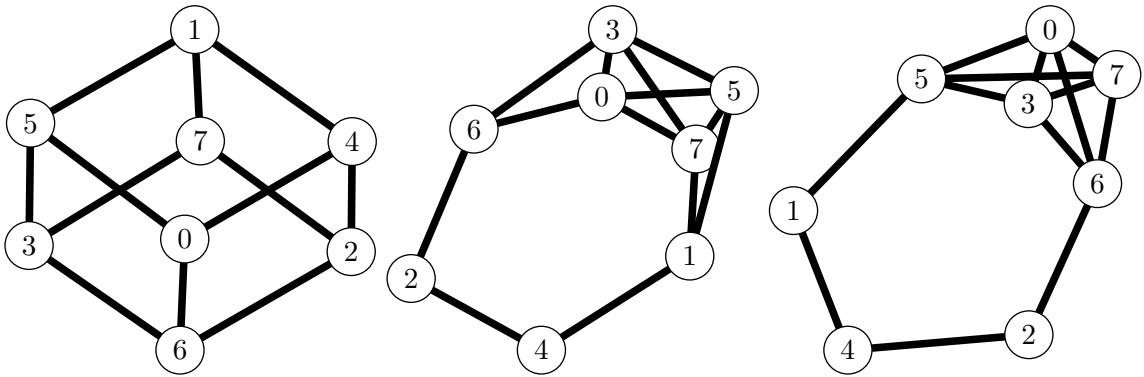
\begin{figure}[h]

	\begin{tikzpicture}[scale=0.85]
	\definecolor{cv0}{rgb}{0.0,0.0,0.0}
	\definecolor{cfv0}{rgb}{1.0,1.0,1.0}
	\definecolor{clv0}{rgb}{0.0,0.0,0.0}
	\definecolor{cv1}{rgb}{0.0,0.0,0.0}
	\definecolor{cfv1}{rgb}{1.0,1.0,1.0}
	\definecolor{clv1}{rgb}{0.0,0.0,0.0}
	\definecolor{cv2}{rgb}{0.0,0.0,0.0}
	\definecolor{cfv2}{rgb}{1.0,1.0,1.0}
	\definecolor{clv2}{rgb}{0.0,0.0,0.0}
	\definecolor{cv3}{rgb}{0.0,0.0,0.0}
	\definecolor{cfv3}{rgb}{1.0,1.0,1.0}
	\definecolor{clv3}{rgb}{0.0,0.0,0.0}
	\definecolor{cv4}{rgb}{0.0,0.0,0.0}
	\definecolor{cfv4}{rgb}{1.0,1.0,1.0}
	\definecolor{clv4}{rgb}{0.0,0.0,0.0}
	\definecolor{cv5}{rgb}{0.0,0.0,0.0}
	\definecolor{cfv5}{rgb}{1.0,1.0,1.0}
	\definecolor{clv5}{rgb}{0.0,0.0,0.0}
	\definecolor{cv6}{rgb}{0.0,0.0,0.0}
	\definecolor{cfv6}{rgb}{1.0,1.0,1.0}
	\definecolor{clv6}{rgb}{0.0,0.0,0.0}
	\definecolor{cv7}{rgb}{0.0,0.0,0.0}
	\definecolor{cfv7}{rgb}{1.0,1.0,1.0}
	\definecolor{clv7}{rgb}{0.0,0.0,0.0}
	\definecolor{cv0v4}{rgb}{0.0,0.0,0.0}
	\definecolor{cv0v5}{rgb}{0.0,0.0,0.0}
	\definecolor{cv0v6}{rgb}{0.0,0.0,0.0}
	\definecolor{cv1v4}{rgb}{0.0,0.0,0.0}
	\definecolor{cv1v5}{rgb}{0.0,0.0,0.0}
	\definecolor{cv1v7}{rgb}{0.0,0.0,0.0}
	\definecolor{cv2v4}{rgb}{0.0,0.0,0.0}
	\definecolor{cv2v6}{rgb}{0.0,0.0,0.0}
	\definecolor{cv2v7}{rgb}{0.0,0.0,0.0}
	\definecolor{cv3v5}{rgb}{0.0,0.0,0.0}
	\definecolor{cv3v6}{rgb}{0.0,0.0,0.0}
	\definecolor{cv3v7}{rgb}{0.0,0.0,0.0}
	\Vertex[style={minimum size=1.0cm,draw=cv0,fill=cfv0,text=clv0,shape=circle},LabelOut=false,L=\hbox{$0$},x=2.4086cm,y=1.7328cm]{v0}
	\Vertex[style={minimum size=1.0cm,draw=cv1,fill=cfv1,text=clv1,shape=circle},LabelOut=false,L=\hbox{$1$},x=2.5642cm,y=5.0cm]{v1}
	\Vertex[style={minimum size=1.0cm,draw=cv2,fill=cfv2,text=clv2,shape=circle},LabelOut=false,L=\hbox{$2$},x=4.9844cm,y=1.5364cm]{v2}
	\Vertex[style={minimum size=1.0cm,draw=cv3,fill=cfv3,text=clv3,shape=circle},LabelOut=false,L=\hbox{$3$},x=0.0cm,y=1.6318cm]{v3}
	\Vertex[style={minimum size=1.0cm,draw=cv4,fill=cfv4,text=clv4,shape=circle},LabelOut=false,L=\hbox{$4$},x=5.0cm,y=3.2536cm]{v4}
	\Vertex[style={minimum size=1.0cm,draw=cv5,fill=cfv5,text=clv5,shape=circle},LabelOut=false,L=\hbox{$5$},x=0.0237cm,y=3.5408cm]{v5}
	\Vertex[style={minimum size=1.0cm,draw=cv6,fill=cfv6,text=clv6,shape=circle},LabelOut=false,L=\hbox{$6$},x=2.3354cm,y=0.0cm]{v6}
	\Vertex[style={minimum size=1.0cm,draw=cv7,fill=cfv7,text=clv7,shape=circle},LabelOut=false,L=\hbox{$7$},x=2.6499cm,y=3.2628cm]{v7}
	\Edge[lw=0.1cm,style={color=cv0v4,},](v0)(v4)
	\Edge[lw=0.1cm,style={color=cv0v5,},](v0)(v5)
	\Edge[lw=0.1cm,style={color=cv0v6,},](v0)(v6)
	\Edge[lw=0.1cm,style={color=cv1v4,},](v1)(v4)
	\Edge[lw=0.1cm,style={color=cv1v5,},](v1)(v5)
	\Edge[lw=0.1cm,style={color=cv1v7,},](v1)(v7)
	\Edge[lw=0.1cm,style={color=cv2v4,},](v2)(v4)
	\Edge[lw=0.1cm,style={color=cv2v6,},](v2)(v6)
	\Edge[lw=0.1cm,style={color=cv2v7,},](v2)(v7)
	\Edge[lw=0.1cm,style={color=cv3v5,},](v3)(v5)
	\Edge[lw=0.1cm,style={color=cv3v6,},](v3)(v6)
	\Edge[lw=0.1cm,style={color=cv3v7,},](v3)(v7)
	\end{tikzpicture}
	\begin{tikzpicture}[scale=0.85]
	\definecolor{cv0}{rgb}{0.0,0.0,0.0}
	\definecolor{cfv0}{rgb}{1.0,1.0,1.0}
	\definecolor{clv0}{rgb}{0.0,0.0,0.0}
	\definecolor{cv1}{rgb}{0.0,0.0,0.0}
	\definecolor{cfv1}{rgb}{1.0,1.0,1.0}
	\definecolor{clv1}{rgb}{0.0,0.0,0.0}
	\definecolor{cv2}{rgb}{0.0,0.0,0.0}
	\definecolor{cfv2}{rgb}{1.0,1.0,1.0}
	\definecolor{clv2}{rgb}{0.0,0.0,0.0}
	\definecolor{cv3}{rgb}{0.0,0.0,0.0}
	\definecolor{cfv3}{rgb}{1.0,1.0,1.0}
	\definecolor{clv3}{rgb}{0.0,0.0,0.0}
	\definecolor{cv4}{rgb}{0.0,0.0,0.0}
	\definecolor{cfv4}{rgb}{1.0,1.0,1.0}
	\definecolor{clv4}{rgb}{0.0,0.0,0.0}
	\definecolor{cv5}{rgb}{0.0,0.0,0.0}
	\definecolor{cfv5}{rgb}{1.0,1.0,1.0}
	\definecolor{clv5}{rgb}{0.0,0.0,0.0}
	\definecolor{cv6}{rgb}{0.0,0.0,0.0}
	\definecolor{cfv6}{rgb}{1.0,1.0,1.0}
	\definecolor{clv6}{rgb}{0.0,0.0,0.0}
	\definecolor{cv7}{rgb}{0.0,0.0,0.0}
	\definecolor{cfv7}{rgb}{1.0,1.0,1.0}
	\definecolor{clv7}{rgb}{0.0,0.0,0.0}
	\definecolor{cv0v3}{rgb}{0.0,0.0,0.0}
	\definecolor{cv0v5}{rgb}{0.0,0.0,0.0}
	\definecolor{cv0v6}{rgb}{0.0,0.0,0.0}
	\definecolor{cv0v7}{rgb}{0.0,0.0,0.0}
	\definecolor{cv1v4}{rgb}{0.0,0.0,0.0}
	\definecolor{cv1v5}{rgb}{0.0,0.0,0.0}
	\definecolor{cv1v7}{rgb}{0.0,0.0,0.0}
	\definecolor{cv2v4}{rgb}{0.0,0.0,0.0}
	\definecolor{cv2v6}{rgb}{0.0,0.0,0.0}
	\definecolor{cv3v5}{rgb}{0.0,0.0,0.0}
	\definecolor{cv3v6}{rgb}{0.0,0.0,0.0}
	\definecolor{cv3v7}{rgb}{0.0,0.0,0.0}
	\definecolor{cv5v7}{rgb}{0.0,0.0,0.0}
	\Vertex[style={minimum size=1.0cm,draw=cv0,fill=cfv0,text=clv0,shape=circle},LabelOut=false,L=\hbox{$0$},x=2.9511cm,y=3.9498cm]{v0}
	\Vertex[style={minimum size=1.0cm,draw=cv1,fill=cfv1,text=clv1,shape=circle},LabelOut=false,L=\hbox{$1$},x=4.3115cm,y=1.4668cm]{v1}
	\Vertex[style={minimum size=1.0cm,draw=cv2,fill=cfv2,text=clv2,shape=circle},LabelOut=false,L=\hbox{$2$},x=0.0cm,y=1.1174cm]{v2}
	\Vertex[style={minimum size=1.0cm,draw=cv3,fill=cfv3,text=clv3,shape=circle},LabelOut=false,L=\hbox{$3$},x=3.1181cm,y=5.0cm]{v3}
	\Vertex[style={minimum size=1.0cm,draw=cv4,fill=cfv4,text=clv4,shape=circle},LabelOut=false,L=\hbox{$4$},x=2.0158cm,y=0.0cm]{v4}
	\Vertex[style={minimum size=1.0cm,draw=cv5,fill=cfv5,text=clv5,shape=circle},LabelOut=false,L=\hbox{$5$},x=5.0cm,y=4.0497cm]{v5}
	\Vertex[style={minimum size=1.0cm,draw=cv6,fill=cfv6,text=clv6,shape=circle},LabelOut=false,L=\hbox{$6$},x=0.9727cm,y=3.4471cm]{v6}
	\Vertex[style={minimum size=1.0cm,draw=cv7,fill=cfv7,text=clv7,shape=circle},LabelOut=false,L=\hbox{$7$},x=4.4079cm,y=3.1373cm]{v7}
	\Edge[lw=0.1cm,style={color=cv0v3,},](v0)(v3)
	\Edge[lw=0.1cm,style={color=cv0v5,},](v0)(v5)
	\Edge[lw=0.1cm,style={color=cv0v6,},](v0)(v6)
	\Edge[lw=0.1cm,style={color=cv0v7,},](v0)(v7)
	\Edge[lw=0.1cm,style={color=cv1v4,},](v1)(v4)
	\Edge[lw=0.1cm,style={color=cv1v5,},](v1)(v5)
	\Edge[lw=0.1cm,style={color=cv1v7,},](v1)(v7)
	\Edge[lw=0.1cm,style={color=cv2v4,},](v2)(v4)
	\Edge[lw=0.1cm,style={color=cv2v6,},](v2)(v6)
	\Edge[lw=0.1cm,style={color=cv3v5,},](v3)(v5)
	\Edge[lw=0.1cm,style={color=cv3v6,},](v3)(v6)
	\Edge[lw=0.1cm,style={color=cv3v7,},](v3)(v7)
	\Edge[lw=0.1cm,style={color=cv5v7,},](v5)(v7)
	\end{tikzpicture}
	\begin{tikzpicture}[scale=0.85]
	\definecolor{cv0}{rgb}{0.0,0.0,0.0}
	\definecolor{cfv0}{rgb}{1.0,1.0,1.0}
	\definecolor{clv0}{rgb}{0.0,0.0,0.0}
	\definecolor{cv1}{rgb}{0.0,0.0,0.0}
	\definecolor{cfv1}{rgb}{1.0,1.0,1.0}
	\definecolor{clv1}{rgb}{0.0,0.0,0.0}
	\definecolor{cv2}{rgb}{0.0,0.0,0.0}
	\definecolor{cfv2}{rgb}{1.0,1.0,1.0}
	\definecolor{clv2}{rgb}{0.0,0.0,0.0}
	\definecolor{cv3}{rgb}{0.0,0.0,0.0}
	\definecolor{cfv3}{rgb}{1.0,1.0,1.0}
	\definecolor{clv3}{rgb}{0.0,0.0,0.0}
	\definecolor{cv4}{rgb}{0.0,0.0,0.0}
	\definecolor{cfv4}{rgb}{1.0,1.0,1.0}
	\definecolor{clv4}{rgb}{0.0,0.0,0.0}
	\definecolor{cv5}{rgb}{0.0,0.0,0.0}
	\definecolor{cfv5}{rgb}{1.0,1.0,1.0}
	\definecolor{clv5}{rgb}{0.0,0.0,0.0}
	\definecolor{cv6}{rgb}{0.0,0.0,0.0}
	\definecolor{cfv6}{rgb}{1.0,1.0,1.0}
	\definecolor{clv6}{rgb}{0.0,0.0,0.0}
	\definecolor{cv7}{rgb}{0.0,0.0,0.0}
	\definecolor{cfv7}{rgb}{1.0,1.0,1.0}
	\definecolor{clv7}{rgb}{0.0,0.0,0.0}
	\definecolor{cv0v3}{rgb}{0.0,0.0,0.0}
	\definecolor{cv0v5}{rgb}{0.0,0.0,0.0}
	\definecolor{cv0v6}{rgb}{0.0,0.0,0.0}
	\definecolor{cv0v7}{rgb}{0.0,0.0,0.0}
	\definecolor{cv1v4}{rgb}{0.0,0.0,0.0}
	\definecolor{cv1v5}{rgb}{0.0,0.0,0.0}
	\definecolor{cv2v4}{rgb}{0.0,0.0,0.0}
	\definecolor{cv2v6}{rgb}{0.0,0.0,0.0}
	\definecolor{cv3v5}{rgb}{0.0,0.0,0.0}
	\definecolor{cv3v6}{rgb}{0.0,0.0,0.0}
	\definecolor{cv3v7}{rgb}{0.0,0.0,0.0}
	\definecolor{cv5v7}{rgb}{0.0,0.0,0.0}
	\definecolor{cv6v7}{rgb}{0.0,0.0,0.0}
	\Vertex[style={minimum size=1.0cm,draw=cv0,fill=cfv0,text=clv0,shape=circle},LabelOut=false,L=\hbox{$0$},x=3.9694cm,y=5.0cm]{v0}
	\Vertex[style={minimum size=1.0cm,draw=cv1,fill=cfv1,text=clv1,shape=circle},LabelOut=false,L=\hbox{$1$},x=0.0cm,y=2.1763cm]{v1}
	\Vertex[style={minimum size=1.0cm,draw=cv2,fill=cfv2,text=clv2,shape=circle},LabelOut=false,L=\hbox{$2$},x=3.6415cm,y=0.2403cm]{v2}
	\Vertex[style={minimum size=1.0cm,draw=cv3,fill=cfv3,text=clv3,shape=circle},LabelOut=false,L=\hbox{$3$},x=3.6324cm,y=3.8451cm]{v3}
	\Vertex[style={minimum size=1.0cm,draw=cv4,fill=cfv4,text=clv4,shape=circle},LabelOut=false,L=\hbox{$4$},x=0.8417cm,y=0.0cm]{v4}
	\Vertex[style={minimum size=1.0cm,draw=cv5,fill=cfv5,text=clv5,shape=circle},LabelOut=false,L=\hbox{$5$},x=1.9816cm,y=4.231cm]{v5}
	\Vertex[style={minimum size=1.0cm,draw=cv6,fill=cfv6,text=clv6,shape=circle},LabelOut=false,L=\hbox{$6$},x=4.7032cm,y=2.5969cm]{v6}
	\Vertex[style={minimum size=1.0cm,draw=cv7,fill=cfv7,text=clv7,shape=circle},LabelOut=false,L=\hbox{$7$},x=5.0cm,y=4.2992cm]{v7}
	\Edge[lw=0.1cm,style={color=cv0v3,},](v0)(v3)
	\Edge[lw=0.1cm,style={color=cv0v5,},](v0)(v5)
	\Edge[lw=0.1cm,style={color=cv0v6,},](v0)(v6)
	\Edge[lw=0.1cm,style={color=cv0v7,},](v0)(v7)
	\Edge[lw=0.1cm,style={color=cv1v4,},](v1)(v4)
	\Edge[lw=0.1cm,style={color=cv1v5,},](v1)(v5)
	\Edge[lw=0.1cm,style={color=cv2v4,},](v2)(v4)
	\Edge[lw=0.1cm,style={color=cv2v6,},](v2)(v6)
	\Edge[lw=0.1cm,style={color=cv3v5,},](v3)(v5)
	\Edge[lw=0.1cm,style={color=cv3v6,},](v3)(v6)
	\Edge[lw=0.1cm,style={color=cv3v7,},](v3)(v7)
	\Edge[lw=0.1cm,style={color=cv5v7,},](v5)(v7)
	\Edge[lw=0.1cm,style={color=cv6v7,},](v6)(v7)

	\end{tikzpicture}

	\caption{The three extremal graphs for $r=3$ and $n=8$: $Q_3, G_{8,3,2}$ and $ G_{8,3,1}$ }
	\label{fig:extremaln8r3}
\end{figure}


In the digraph setting, the total distance $W(D)$ of a digraph $D$ equals the sum of all distances between all ordered pairs of vertices. The outradius of $D$ is equal to the smallest value $r$ such that there exists a vertex $x$ for which $d(x,v)\le r$ for every vertex $v$ of the digraph $D$.

In this setting, when the outradius $r=1$, the extremal digraphs are obviously bidirected cliques, their total distance being $2\binom{n}{2}.$
When $r=2$, the extremal digraphs are bidirected cliques missing $n$ edges, one starting in every vertex, with the restriction that no $n-1$ of those missing edges end in the same vertex. In this case the total distance $W(D)=n^2.$
For $r\ge 3$, we propose the digraph analog to Conjecture~\ref{conjchen}. The definition of $D_{n,r,s}$ being stated in Section~\ref{not&def}. Figure~\ref{fig:Dnrs} shows $D_{n,r,s}$ for $r=3.$
%
\begin{conj}\label{conjchendigraph}
	Let $n$ and $r$ be two positive integers with $n \ge 2r$ and $r \ge 3$. For any digraph $D$ of order $n$ with outradius $r$, $W(D) \ge W(D_{n,r,1})$. Equality holds if and only if $D \cong D_{n,r,s}$ for $1 \le s \le \frac{n-2r+2}{2}$.
\end{conj}

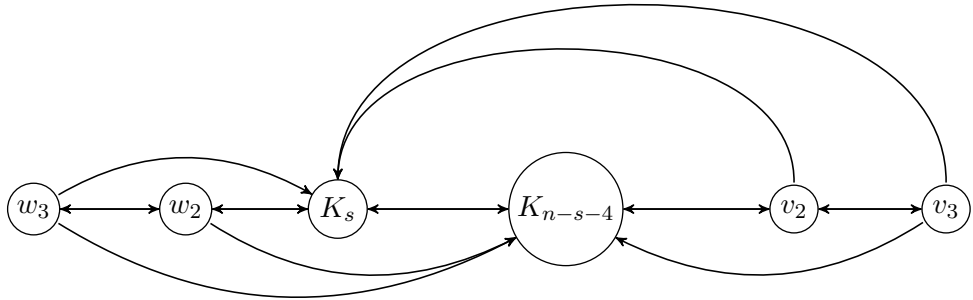
\begin{figure}[h]
	\centering

	\begin{tikzpicture}
	\definecolor{cv0}{rgb}{0.0,0.0,0.0}
	\definecolor{c}{rgb}{1.0,1.0,1.0}

	\Vertex[L=\hbox{$v_2$},x=13,y=0]{v1}
	\Vertex[L=\hbox{$v_3$},x=15,y=0]{v2}

	\Vertex[L=\hbox{$w_2$},x=5,y=0]{w1}
	\Vertex[L=\hbox{$w_3$},x=3,y=0]{w2}
	\Vertex[L=\hbox{$K_s$},x=7,y=0]{w0}
	
	\Vertex[L=\hbox{$K_{n-s-4}$},x=10cm,y=0.0cm]{v0}
	
	\Edge[lw=0.1cm,style={post, right}](v0)(v1)
	\Edge[lw=0.1cm,style={post, right}](v1)(v2)
	\Edge[lw=0.1cm,style={post, right}](w0)(v0)
	\Edge[lw=0.1cm,style={post, right}](v1)(v0)
	\Edge[lw=0.1cm,style={post, right}](v2)(v1)
	\Edge[lw=0.1cm,style={post, right}](v0)(w0)
	\Edge[lw=0.1cm,style={post, right}](w1)(w2)
	\Edge[lw=0.1cm,style={post, right}](w0)(w1)
	\Edge[lw=0.1cm,style={post, right}](w1)(w0)
	\Edge[lw=0.1cm,style={post, right}](w2)(w1)

	\Edge[lw=0.1cm,style={post, bend left}](w2)(w0)
	\Edge[lw=0.1cm,style={post, bend right}](w2)(v0)
	\Edge[lw=0.1cm,style={post, bend right}](w1)(v0)
	
		\Edge[lw=0.1cm,style={post, bend right=90}](v2)(w0)
	\Edge[lw=0.1cm,style={post, bend left}](v2)(v0)
	\Edge[lw=0.1cm,style={post, bend right=90}](v1)(w0)

	\end{tikzpicture}

	\caption{The digraph $D_{n,r,s}$ for $r=3$}
	\label{fig:Dnrs}
\end{figure}

Just as in the graph case, like with Conjecture~\ref{conjchen}, for fixed $r$ there may be a few counterexamples for small $n$.
Also Conjecture~\ref{conjchendigraph} is asymptotically true.

\begin{thr}\label{maindi}
	For $r\ge 3$, there exists a value $n_1(r)$ such that for all $n \ge n_1(r)$ the following hold
	\begin{itemize}
		\item for any digraph $D$ of order $n$ with outradius $r$, we have $W(D) \ge W(D_{n,r,1})$. Equality holds if and only if $D \cong D_{n,r,s}$ for $1 \le s \le \frac{n-2r+2}{2}$.
	\end{itemize}
\end{thr}

The proof for this in Section~\ref{AsProofChendigraph} uses the same main ideas as in the graph case, but turns out to be slightly more difficult as the distance function is not symmetric. In this case, we consider an expression which gives information about $W(D)-W(D \backslash v)$ in a more general setting. In Section~\ref{subsec:Min_givenrad} we start the investigation of the question for radius instead of out‐radius in the digraph case.

\section{Notation and definitions}\label{not&def}

A graph will be denoted by $G=(V,E)$ and 
a digraph will be denoted by $D=(V,A).$
The order $\lvert V \rvert$ will be denoted by $n$. 
A clique or bidirected clique on $n$ vertices will be denoted by $K_n$. 
A cycle or directed cycle of length $k$ will be denoted by $C_k$.
The clique number of a graph $G$, $\omega(G)$, is the order of the largest clique which is a subgraph of $G.$
The complement $G^c$ of graph $G=(V,E)$ is the graph with vertex set $V$ and edge set $E^c=\binom{V}{2} \backslash E.$ 
The complement $D^c$ of a digraph $D$ is defined similarly, where the set of directed edges is the complement with respect to the edges of a bidirected clique.

The degree of a vertex in a graph $\deg(v)$ equals the number of neighbours of the vertex $v$, i.e. $\deg(v)=\lvert N(v) \rvert.$
In a digraph, we denote with $N^-(v)$ and $N^+(v)$ the (open) in- and outneighbourhood of a vertex $v$.
The indegree $\deg^-$ and outdegree $\deg^+$ of a vertex $v$, equals the number of arrows ending in or starting from the vertex $v$, i.e. $\deg^+(v)=\lvert N^+(v) \rvert $ and  $\deg^-(v)=\lvert N^-(v) \rvert $.
The total degree $\deg$ of a vertex $v$ in a digraph is the sum of the in- and outdegree, i.e. $\deg(v)=\deg^+(v)+\deg^-(v).$

Let $d(u, v)$ denote the distance between vertices $u$ and $v$ in a graph $G$ or digraph $D$, i.e. the number of edges or arrows in a shortest path from $u$ to $v$.  
The eccentricity of a vertex $v$ in a graph equals $\ecc(v)=d(v,V)=\max_{u \in V} d(v,u).$
The radius and diameter of a graph on vertex set $V$ are respectively equal to $\min_{v \in V} \ecc(v)$ and $\max_{v \in V} \ecc(v)=\max_{u,v \in V} d(u,v).$

In the case of digraphs, the distance function between vertices is not symmetric and so there is a difference between the inner- and outer eccentricity 
$\ecc^- (v)=d(V,v)=\max_{u \in V} d(u,v)$ and $\ecc^+(v)=d(v,V)=\max_{u \in V} d(v,u).$
We use the conventions as in e.g. \cite{JG}.
Radius, inradius and outradius are defined in Subsection $2.1$ in \cite{JG} or Subsection $3.1$ in \cite{JG2} but for clarity we define them in the next sentences.
The in- and outradius of a digraph $D$ are defined by 
$\rad^{-}(D)= \min\{d(V,x)\mid x \in V\}$ and $\rad^{+}(D)= \min\{d(x,V)\mid x \in V\}$.
The radius of a digraph $D$ is defined as $\rad(D)=\min\{ \frac{d(x,V)+ d(V,x)}2 \mid x \in V \}.$ Sometimes authors refer to the outradius as radius (see e.g. Subsection $1.4$ in~\cite{CLZ}), as outradius is the most common one between those three definitions.

The total distance, also called the Wiener index, of a graph $G$ equals the sum of distances between all unordered pairs of vertices, i.e. $W(G)=\sum_{\{u,v\} \subset V} d(u,v).$ 
The average distance of a graph is $\mu(G)=\frac{W(G)}{\binom{n}{2}}$. 
The Wiener index of a digraph equals the sum of distances between all ordered pairs of vertices, i.e. $W(D)=\sum_{(u,v) \in V^2} d(u,v).$
The average distance of the digraph is $\mu(D)=\frac{W(D)}{n^2-n}.$
A digraph is called biconnected if $d(u,v)$ is finite for any $2$ vertices $u$ and $v$. 

The statement $f(x)=O(g(x))$ as $x \to \infty$ implies that there exist fixed constants $x_0, M>0$, such that for all $x \ge x_0$ we have $\lvert f(x) \rvert \le M \lvert g(x) \rvert .$
Analogously, $f(x)=\Omega(g(x))$ as $x \to \infty$ implies that there exist fixed constants $x_0, M>0$, such that for all $x \ge x_0$ we have $\lvert f(x) \rvert \ge M \lvert g(x) \rvert.$
If $f(x)=\Omega(g(x))$ and $f(x)=O(g(x))$ as $x \to \infty$, then one uses $f(x)=\Theta(g(x))$ as $x \to \infty$. Sometimes we do not write the "as $x \to \infty$" if the context is clear.

\begin{defi}
	Given a graph $G$ and a vertex $v$, the blow-up of a vertex $v$ of a graph $G$ by a graph $H$ is constructed as follows.
	Take $G \backslash v$ and connect all initial neighbours of $v$ with all vertices of a copy of $H.$	
	When taking the blow-up of a vertex $v$ of a digraph $D$ by a digraph $H$, a directed edge between a vertex $w$ of $D \backslash v$ and a vertex $z$ of $H$ is drawn if and only if initially there was a directed edge between $w$ and $v$ in the same direction. 
	When taking the blow-ups of multiple vertices, for neighbouring vertices $v_1$ and $v_2$, all vertices of the corresponding graphs $H_1$ and $H_2$ are connected as well in the blow-up (possible in one direction in the digraph case). Equivalently, one can take the blow-up of the different vertices one at a time at the resulting graph of the blow-up in the previous step.  

\end{defi}

Let $G_{n,r,s}$, where $n \ge 2r$ and $1 \le s \le \frac{n-2r+2}{2}$, be the graph obtained by taking two blow-ups of two consecutive vertices in a cycle $C_{2r}$ by cliques $K_s$ and $K_{n-2r+2-s}$ respectively. 
Note that $\omega(G_{n,r,s})=n-2r+2.$

Let $D_{2r,r,1}$ be a digraph with $2r$ vertices $v_1,v_2, \ldots v_r$ and $w_1, w_2, \ldots, w_r$, such that there are directed edges from $v_i$ to $v_j$ and from $w_i$ to $w_j$ if and only if $j \le i+1$ and a directed edge from any $v_i$ to $w_1$ and from any $w_i$ to $v_1.$

Let $D_{n,r,s}$, $n \ge 2r$ and $1 \le s \le \frac{n-2r+2}{2}$, be the digraph obtained by taking the blow-up of $v_1$ by a bidirected clique $K_s$ and a blow-up of $w_1$ by a bidirected clique $K_{n-2r+2-s}$.

\section{Conjecture~\ref{conjchen} for large order}\label{AsProofChen}
In this section we will prove the conjecture of Chen et al, Conjecture~\ref{conjchen}, for sufficiently large order compared with the radius. The main idea in the proof is that in a possible extremal graph, there is a big blow-up (clique) and these vertices add at least some value to the total distance, where equality only holds if we have some structure close to our conjectured extremal graphs. More detailed, in this section we start with the calculation of the total distance of the graphs $G_{n,r,s}$ and $D_{n,r,s}$, which we want to prove to be extremal. This is done in Lemma~\ref{lem:compu}. Next, in Lemma~\ref{lem1} we show that graphs with small total distance contain a big clique.
In Lemma~\ref{lem2} we prove that when the total distance is small, there is at least one vertex which does not influence other distances nor the radius.
Lemma~\ref{lem3} then describes the minimal sum of distances using such a vertex, where equality happens if and only if we have a certain base structure. To finish the proof, we show that if the order is large, equality in Lemma~\ref{lem3} has to occur in some step. From that, we derive a lowerbound of the total distance which equals the total distance of the graphs of the form $G_{n,r,s}$ and we conclude.

\begin{lem}\label{lem:compu}
		For every $r \ge 2$, $n \ge 2r$ and $1 \le s \le \frac{n-2r+2}{2}$, we have
	\begin{equation}\label{eq:1}
	W(G_{n,r,s})=\binom{n}{2}+(r-1)^2n -r(r-1)^2
	\end{equation}
	\begin{equation}\label{eq:2}
	W(D_{n,r,s})=2\binom{n}{2}+(r-1)^2 n-4 \binom{r}{3}.
	\end{equation}
	
\end{lem}

\begin{proof}
	The graph $G_{n,r,s}$ can be constructed in a different way.
	Start with a cycle $C_{2r}=(v_1v_2 \ldots v_{2r})$ and a $K_{n-2r}$.
	Connect $s-1$ vertices of $K_{n-2r}$ with $v_1, v_2$ and $v_3$, and the other $n-2r-(s-1)$ vertices with $v_2, v_3$ and $v_4$.
	The sum of distances between any vertex of the cycle $C_{2r}$ with respect to all other vertices of the cycle equals $1+2+\ldots +r +(r-1)+\ldots +1 = r^2$. Due to double counting, one has $W(C_{2r})=\frac{2r \times r^2}{2}=r^3.$
	The sum of distances between a vertex of the $K_{n-2r}$ and all vertices of $C_{2r}$ is now exactly one more, i.e. $r^2+1$.
	Due to these observations, we can now compute $W(G_{n,r,s}).$
	
	\begin{align*}
	W(G_{n,r,s})&=W(C_{2r})+W(K_{n-2r})+\sum_{u \in C_{2r}, v \in K_{n-2r}} d(u,v)\\
	&=r^3+\binom{n-2r}{2}+(n-2r)(r^2+1)\\
	&=\binom{n}{2}+(r-1)^2n -r(r-1)^2
	\end{align*}
	
	In the digraph case, one can build $D_{n,r,s}$ from a $D_{2r,r,1}$ and a $K_{n-2r}$ with some additional edges.
	First, we will compute $W(D_{2r,r,1})$ where $v_i, w_i$ are used as has been done in the definition of $D_{2r,r,1}$ at the end of Section~\ref{not&def}.
	The factors $2$ are due to symmetry between $v$ and $w$.
	The factor $r$ in the second term is due to the fact that $d(v_i,w_j)=j$ for every $1 \le i \le r.$
	\begin{align*}
	W(D_{2r,r,1})&=2\sum_{i<j} \left(d(v_i,v_j)+d(v_j,v_i) \right) + 
	2\sum_{1 \le i,j \le r} d(v_i,w_j)\\
	&=2\sum_{1\le i<j\le r} \left((j-i)+1 \right) + 2r \sum_{1 \le j \le r} j\\
	&=2\sum_{2\le j\le r}\left(\binom{j+1}{2}-1\right) + r^2(r+1) \\
	&=2\binom{r+2}{3}-2r+ r^2(r+1)\\
	\end{align*}
	
	
	Also we note that for every $v \in K_{n-2r}$, we have $\sum_{u \in D_{2r,r,1}} d(u,v)=2r$ and
	$$\sum_{u \in D_{2r,r,1}} d(u,v)=1+\left(1+2+\ldots + (r-1)\right)+\left(1+2+\ldots +r\right)=r^2+1.$$
	
	Using these computations, we can now compute $W(D_{n,r,s})$.
	\begin{align*}
	W(D_{n,r,s})&=W(K_{n-2r})+W(D_{2r,r,1})+\sum_{v \in K_{n-2r},u \in D_{2r,r,1}} \left(d(u,v)+d(v,u) \right)\\
	&=(n-2r)(n-2r-1)+2\binom{r+2}{3}-2r+ r^2(r+1)
	+(n-2r)(r+1)^2\\
	&=n^2+r(r-2)n-\frac{2r(r-1)(r-2)}3\\
	&=2\binom{n}{2}+(r-1)^2 n-4 \binom{r}{3}.  \qedhere
	\end{align*}
\end{proof}

\begin{lem}\label{lem1}
	Suppose $G$ is a graph with order $n$ and total distance $W(G) < \binom{n}{2}+an$, for some positive constant $a$.
	Then $\omega(G) \ge \frac{n}{8a},$ i.e. $G$ contains a clique of order at least $\frac{n}{8a}$. Furthermore there exists such a clique such that all its vertices have degree at least equal to $n-4a.$
\end{lem}

\begin{proof}
	Let $S$ be the set of vertices of degree at least $n-4a,$ equivalently the set of vertices for which strictly less than $4a$ vertices are at distance at least $2$ from $v$. 
	Note that $\lvert S \rvert \ge \frac n2$ since otherwise for at least half of the vertices $v$, we have that $\sum_{u \in V} d(u,v) \ge (n-1) + 4a$ and hence 
	$$2W(G)=\sum_{v \in V} \sum_{u \in V} d(u,v) \ge n(n-1)+\frac{n}{2}4a=2\left(\binom{n}{2}+an\right).$$
	Now the following algorithm (presented in pseudocode) returns a set $T$ of at least $\frac{n}{8a}$ vertices in $S$ which form a clique, as we will prove next.
\begin{lstlisting}
Start with $T=\emptyset$ and $U=S.$
While $U$ is nonempty, do:
	Take arbitrary $v \in U$ and set $T$:=$T \cup \{v\}$ and $U$:=$U \cap N(v)$
return $T$
\end{lstlisting}
By induction, we see that every vertex in $U$ is a common neighbour of all vertices in $T$ and $T$ is a clique. The base case is trivial since $T= \emptyset$.
In every step, when a new vertex $v$ is chosen, the vertices in $T \cup \{v\}$ form a clique since we add a vertex adjacent to all vertices of a clique.
The set $U$, only containing common neighbours of the original set $T$ by the induction hypothesis, is updated by removing the non-neighbours of $v$. Hence $U \cap N(v)$ only contains common neighbours of $T \cup \{v\}$. So the induction is done.
In every step, $\lvert T \rvert $ increases by one, while $\lvert U \rvert $ decreases by at most $4a$ since $N(v)$ contains all vertices except at most $4a$ by the definition of $S$.
This implies that the algorithm cannot terminate after less than $\frac{|S|}{4a} \ge \frac{n}{8a}$ steps and so $\lvert T \rvert$ is at least this value. This is a clique containing only vertices of $S$, i.e. all of its vertices have degree being at least $n-4a.$
\end{proof}

\begin{lem}\label{lem2}
	Let $a, r \ge 3$ be fixed integers.
	There is a value $n_0(a,r)$ such that for any $n \ge n_0$ and any graph $G$ of order $n$ and radius $r$ with $W(G)<\binom{n}{2}+a n$, there is a vertex $v \in G$ such that $G \backslash v$ has radius $r$ and the distance between any $2$ vertices of $G \backslash v$ equals the distance between them in $G.$
\end{lem}

\begin{proof}
	Let $n_0:=n_0(a, r)= 192a^3$. By Lemma~\ref{lem1}, we know that $G$ contains a clique $K_k$ with $k \ge \frac{n}{8a}=24a^2$, such that the degree of all its vertices is at least $n-4a.$
	We can take a subset $S$ of vertices of $K_k$ of size at most $ 16a^2$ such that for any $2$ vertices $v,w$ in $G \backslash K_k$ which are connected with a common neighbour in $K_k$ have a common neighbour in $S$ and any vertex $v$ in $G \backslash K_k$ with a neighbour in $K_k$ has a neighbour in $S.$

	Taking a first vertex $x$ in $K_k$, it has degree at least $n-4a$ and hence it is connected to all vertices except for $s \le 4a-1$ vertices $v_1, v_2, \ldots, v_s$ in $G \backslash K_k.$
	For every of the $s$ vertices $v_i$ which has a neighbour $x_i$ in $K_k$, add a single $x_i$ to $S.$
	For every $v_i$, $1 \le i \le s$, there are at most $4a-1$ vertices which are not neighbours of $x_i$. For any such non-neighbour of $X_i$ which has a common neighbour in $K_k$ with $v_i$, add the common neighbour to $S$.
	We end with a set $S$ of size at most $4a(4a-1)+1<16a^2$ which satisfies the properties.

	Now for any vertex $z \in K_k \backslash S$, the distance measure in $G \backslash z$ equals the restriction to $G \backslash z$ of the distance measure in $G,$ since for every two vertices in $G \backslash z$ there is a shortest path in $G$ between the two vertices that does not contain $z$ by the construction of $S$.
	This also implies that for every $u \in K_k$ different from $z$, $\ecc(u)$ attains the same value in $G$ and in $G \backslash z$. Vertices $w \in G \backslash K_k$ cannot have larger eccentricity in $G \backslash z$ than in $G.$
	As a consequence, there is at most one value $z^*$ such that $G \backslash z^*$ has radius larger than $r$.
	
	Since removing one vertex cannot decrease the radius with more than one, the radius of $G \backslash z$ is at least $r-1$.
	Equality can occur if and only if there is at least one vertex $w_z$ in $G \backslash K_k$ such that its eccentricity as a vertex in $G \backslash z$ equals $r-1$ and $d_G(z,w_z)=r$.
	Let $T$ be the set of all vertices $z \in K_k \backslash S$ for which this happens. For any fixed $z \in T$, $z$ is the only vertex at distance $r$ from $w_z$, implying that the corresponding $w_z$ are different for different $z \in T.$
	Rewriting $W(G)<\binom{n}{2}+a n$ and using that $d(u,w_z)\ge r-1 \ge 2$ whenever $u \in K_k, z \in T$, we get
	$$an>W(G)-\binom{n}{2} \ge \sum_{u \in K_k, z \in T} \left( d(u,w_z)-1 \right) \ge \lvert T \rvert k.$$
	Since $\frac nk \le 8a$, this implies that $\lvert T \rvert < 8a^2.$
	Combining with $k \ge 24a^2$ and $|S|<16a^2$, we can choose a vertex $v \in K_k \backslash (S \cup T \cup \{z^*\})$ as the latter is non-empty.
	This vertex $v$ is a vertex satisfying the statement of the lemma.
 \end{proof}


\begin{lem}\label{lem3}
	Let $G$ be a graph with radius $r$ and order $n$ such that there is some vertex $v \in G$ such that $G \backslash v$ has also radius $r$ and the same restricted distance function.
	Then \begin{equation}\label{eq:diffW}
	W(G) \ge W(G \backslash v)+n-1+(r-1)^2.
	\end{equation}    
	Equality occurs if and only if there is a path $\Q=w_{r}w_{r-1}\ldots w_1 u_1 u_2 \ldots u_r$ as subgraph of $G$, where $v=w_1$, such that $d_G(w_r,u_1)=d_G(w_1,u_r)=r$ and $d(v,w)=1$ for every vertex $w \in G$ which is not on this path.
\end{lem}

\begin{proof}
	Let the eccentricity of $v$ in $G$ be $r' \ge r.$
	This implies that there exists a path $\P = vu_1u_2 \ldots u_{r'}$ in $G$ which is the shortest path between $u_{r'}$ and $v=u_0$.
	If $r' \ge 2r-1$, then $$W(G)- W(G \backslash v)-(n-1) = \sum_{u \in G,u \not=v} \left( d(u,v)-1 \right) \ge \frac{(r'-1)r'}{2}>(r-1)^2$$ and Inequality~\ref{eq:diffW} holds strictly.
	Since $G$ has radius $r$, there exists a vertex $z$ with $d(u_{r'-r+1},z)=r$.
	Since we now assume $r'<2r-1$, $z$ is not a vertex of $\P.$
	Take a shortest path $\P'$ from $v$ to $z$ and let $u_i$ be the vertex in $\P \cap \P'$ with the largest index, where we take $i=0$ if $\P \cap \P' = v$, i.e. we set $v=u_0.$
	Note that $i<r'-r+1$ since otherwise we have $d(v,z)=d(v,u_{r'-r+1})+d(u_{r'-r+1},z)=r'-r+1 + r=r'+1$, which is a contradiction since	$\ecc(v)=r'.$
	
	By applying the triangle inequality, we have that $d(u_{r'-r+1},u_i)+d(u_i,z)  	\ge d(u_{r'-r+1},z)=r$, i.e. $d(u_i,z) \ge r-(r'-r+1-i)=2r+i-r'-1$.
	We now can estimate $W(G)- W(G \backslash v)-(n-1)$ again.
	
	\begin{align*}
	\sum_{u \in G,u \not=v} \left( d(u,v)-1 \right) &\ge \sum_{j=1}^{r'}\left(  d(v,u_j)-1\right) + \sum_{u \in \P' \backslash \P} \left( d(u,v)-1 \right) \\
	 &\ge \sum_{j=1}^{r'} (j-1) + \sum_{j=1}^{2r+i-r'-1} (i+j-1)\\
	&\ge\frac{(r'-1)r'}{2}+\frac{(2r-r'-2)(2r-r'-1)}{2}\\
	&= \frac{r'^2-2r+1+(2r-r'-1)^2 }{2}\\
	&\ge \frac{(r'-1)^2+\left(2r-2-(r'-1)\right)^2 }{2}\\
	&\ge (r-1)^2.
	\end{align*}
	Here we used $2r-1>r'\ge r$, $i \ge 0$ and the inequality between the quadratic and arithmetic mean (QM-AM). Equality occurs if and only if $r'=r$, $i=0$, $d(v,z)=r-1$ and $d(w,v)=1$ for every vertex $w$ which is not part of $\P$ nor of $\P'$. These conditions are equivalent with the characterization given in the statement of this lemma.
\end{proof}

\begin{proof}[Proof of Theorem~\ref{main}]
	Take $a=a(r)=(r-1)^2$ and $n_1:= n_1(r)=n_0(r)+a(r)n_0(r),$ where $n_0$ is chosen as in Lemma~\ref{lem2}.
	Let $n \ge n_1.$
	Let $G$ be a graph of order $n$ and radius $r$ with minimal total distance among such graphs.
	From the minimum total distance assumption, we have that such a graph $G$ satisfies $W(G)<\binom{n}{2}+a n$, where $a:=a(r)=(r-1)^2$ due to the example $G_{n,r,s}$ and Equation~\ref{eq:1}.
	Let $G=G_n$.
	By iterating Lemma~\ref{lem2} we can find a sequence of $n-n_0 \ge an_0$ vertices $v_{n_0+1}, \ldots, v_{n}$ such that the graphs $G_i=G_{i+1}\backslash v_{i+1}$ for $n_0 \le i \le n$ all have radius $r$, the distance function of $G_j$ equals the distance function of $G_i$ restricted to the vertices of $G_j$ for every $n_0 \le  j<i \le n$ and $W(G_i) \le  \binom{i}{2}+ai$ for every $n_0 \le i \le n$.
	
	The latter expression is true by induction as the base case $W(G)<\binom{n}{2}+a n$ holds and Lemma~\ref{lem3} implies the induction step
	$$W(G_i) \le W(G_{i+1})-i-(r-1)^2<\binom{i+1}{2}+a (i+1) - i -a = \binom{i}{2}+ai.$$

	If there is no equality in any of these steps using Lemma~\ref{lem3}, then 
	
\begin{align*}
	W(G) &\ge W(G_{n_0})+\sum_{i=n_0+1}^{n} \left( i+(r-1)^2 \right)\\
	&> \binom{n_0}{2} + \sum_{i'=n_0}^{n-1} i' +(n-n_0)+(n-n_0)(r-1)^2\\
	&\ge  \binom{n}{2}+an_0+ (n-n_0)a\\
	&= \binom{n}{2}+an
\end{align*}
	which is a contradiction.
	
	Let $v=v_m$ be a vertex whose addition gives equality in Lemma~\ref{lem3},
	then we know part of the characterization of the graph $G_m$ due to Lemma~\ref{lem3}.
	We note that $u_i$ is not connected with $w_j$ when $j<i$ as otherwise $d(w_1,u_i)=j<i$.
	Similarly, $u_i$ is not connected with $w_j$ when $j>i$ as otherwise $d(u_1,w_j)=i<j$.
	Also vertices $u_i$ and $w_i$ are not connected when $1<i \le r-1$ as otherwise $\ecc(w_i)<r$.
	For any neighbour $w$ of $v=w_1$, it is easy to see that we need $N[w] \cap \Q \subset \{w_3,w_2,w_1,u_1,u_2\}$ as otherwise $d(w_1,u_r)$ or $d(u+1,w_r)$ is less than $r$.
	A small case distinction shows that $N[w] \cap \Q$ is part of one of the sets $\{w_3,w_2,w_1\},\{w_2,w_1,u_1\}, \{w_1,u_1,u_2\},$ as otherwise there is some vertex with eccentricity strictly smaller than $r$.
	
	Also if $2$ vertices $x,y$ of $\Q$ satisfy $N[x] \cap \Q=\{w_3,w_2,w_1\}$ and $N[y] \cap \Q= \{w_1,u_1,u_2\}$, they cannot be connected, since otherwise $\ecc(y)<r.$
	Now $\sum_{x,y \in \Q} d(x,y) \ge W(C_{2r})=r^3$,
	$\sum_{x,y \in G_m \backslash \Q} d(x,y) \ge W(K_{m-2r})$ and
	$\sum_{x\in \Q,y \in G_m \backslash \Q} d(x,y) \ge (m-2r)(r^2+1)$.
	
	Using the previous observations, we conclude that the graph with minimal total distance at this point (i.e. equality in the three above estimates) was some $G_{m,r,s}$.
	As $W(G_m)=W(G_{m,r,s})$, we will need equality in Lemma~\ref{lem3} when applying to every $G_k$ and $v_k$ with $n \ge  k > m$ 
	and so $G$ is also of the form $G_{n,r,s}$.
\end{proof}

\section{Conjecture~\ref{conjchendigraph} for large  order}\label{AsProofChendigraph}

In this section, we prove that Conjecture~\ref{conjchendigraph} does also hold when the order is sufficiently large in terms of the outradius.
The main ideas are somewhat similar as in Section~\ref{AsProofChen}.
In Section~\ref{subsec:Min_givenrad} we briefly discuss the analog given the radius instead of outradius of a digraph.


\begin{lem}\label{lem:g1}
	Let $D$ be a digraph with average total degree at least being equal to $2(n-1)-2t$. Then at least half of the vertices have a total degree which is at least $2(n-1)-4t+1.$ 
	Also $\omega(D) \ge \frac{n}{8t},$ i.e. $D$ contains a bidirected clique of size at least $\frac{n}{8t}$ all of whose vertices have total degree more than $2(n-1)-4t.$
\end{lem}

\begin{proof}
	The first part is true as the contrary would lead to a contradiction.
	If more than half of the vertices have total degree equal to at most $2(n-1)-4t$, then the average is smaller than $2(n-1)-2t$ as the other vertices have total degree at most $2(n-1)$.
	The second part is analogous to the proof of Lemma~\ref{lem1}, with $S$ now being the set of vertices with total degree more than $2(n-1)-4t.$
	In the algorithm, $N(v)=N^+(v) \cap N^-(v),$ will denote the set of vertices $w$ which are both in-neighbours and out-neighbours of $v$. Again in every step, at most $4t$ vertices do not belong to both the in- and outneighbourhood of the additional selected vertex, so $U$ is decreased with at most $4t$ and we started with $\lvert S \rvert \ge \frac n2$ vertices.
\end{proof}

\begin{lem}\label{lem:g2}
	Let $D$ be a digraph with outradius $r\ge 3$, order $n$ and size at least $n(n-1-t)$ such that it contains a bidirected clique $K_k$ for which all of its vertices have total degree at least $2(n-1)-4t+1.$ 
	If $k>32t^2+4t+1+\frac{tn}k$, there is a vertex $v \in K_k$ such that $D \backslash v$ has outradius $r$ and the distance between any $2$ vertices of $D \backslash v$ equals the distance between them in $D.$
\end{lem}

\begin{proof}
	We will first construct a set $S$ of vertices of $K_k$ of size at most $32t^2+4t+1$
	such that for any $2$ vertices $x,y$ in $D \backslash K_k$ for which there exists a vertex $v \in K_k$ such that $\vc{xv}$ and $\vc{vy}$ are edges of $D$, there is an $s \in S$ with $\vc{xs}$ and $\vc{sy}$ being edges of $D$ as well. 
	Take a first vertex $s_1$ of $K_k$ and assume $Z$ is the set of vertices $z$ of $D \backslash K_k$ such that there are not edges in both directions between $s_1$ and $z.$ 
	For any vertex $z \in Z$ that has an edge towards $K_k$ and $\vc{zs_1} \not \in A$, take an additional vertex $s_i \in K_k$ (which we put in $S$) such that there is an edge from $z$ to $s_i$. Similarly for every vertex $z \in Z$ such that there is an edge from $K_k$ to $z$ and $\vc{s_1z} \not \in A$, we take some $s_i \in K_k$ for which there is an edge from $s_i$ to $z$.
	Note at this point $\lvert S \rvert \le 4t$ due to the total degree of $s_1$ being at least $2(n-1)-4t+1.$
	Now there at most $2\cdot 4t \cdot 4t=32t^2$ pairs of vertices $(x,y)$ in $D \backslash K_k$ such that the property is not satisfied by $S$ yet.
	Adding a corresponding vertex of $K_k$ to $S$ gives a set $S$ satisfying the property.
	Now for any vertex $v \in K_k \backslash S$, the distance measure in $D \backslash v$ equals the restriction to $D \backslash v$ of the distance measure in $D.$ To see this, pick any $x,y \in D \backslash v$ and note that there is a shortest path from $x$ to $y$ in $D$ that avoids $v$.
	
		As is the case in Lemma~\ref{lem2}, we see there can be at most one vertex $z^* \in K_k \backslash S$ such that the outradius of $D \backslash z^*$ is larger than $r$ (being equal to $r+1$) since the out-eccentricities for the other vertices in $K_k$ are the same in $D$ as in $D \backslash z^*$ and the out-eccentricities for other vertices in $D \backslash K_k$ cannot become larger.

	For any vertex $v \in K_k \backslash S$, the outradius of $D \backslash v$ is at least $r-1$. In case equality holds, there is at least one vertex $w_v$ in $D \backslash K_k$ such that $d(w_v,v)=r$ and the out-eccentricity of $w_v$ as a vertex in $D \backslash v$  equals $r-1$.
	Let $T$ be the set of all vertices $v \in K_k \backslash S$ for which this happens. Note that no value $w$ can be associated with two elements $v_1, v_2 \in T$ because then $d(w,v_1)=d(w,v_2)=r$ and thus the out-eccentricity of $w$ in $D \backslash v_1, D \backslash v_2$ is still at least $r.$
	So for every $v \in T$, there is an associated $w_v$ which is not associated with another element from $T$, for which there is no edge from $w_v$ to any element of $K_k$ since $r \ge 3$.
	This implies that at least $\lvert T \rvert k$ arrows are missing, which has to be at most $tn$ due to the lowerbound on the size $|A|$ of $n(n-1-t)$, i.e. $\lvert T \rvert \le \frac{tn}k.$\\
	Since $k> 32t^2+4t+1+\frac{tn}k \ge \lvert S \rvert + \lvert T \rvert $, we can choose a vertex $v \in K_k \backslash (S \cup T).$
\end{proof}

%
%

\begin{lem}\label{lem2di}
	Let $r \ge 3$.
	There is a value $n_0(r)$ such that for any $n \ge n_0$ and any digraph $D$ of order $n$ and outradius $r$ with $W(D)<2\binom{n}{2}+a n$ for some positive integer $a$, there is a vertex $v \in D$ such that $D \backslash v$ has outradius $r$ and the distance between any $2$ vertices of $D \backslash v$ equals the distance between them in $D.$
\end{lem}

\begin{proof}
	If the size of the digraph would be smaller than $n(n-1-a)$, then at least $an$ pairs of vertices are at distance at least $2$ implying $W(D) \ge n(n-1-a)+ 2an=2\binom{n}{2}+a n$, which is a contradiction. 
	Due to the equivalent of the handshaking lemma, we know the average total degree is at least $2(n-1)-2a$ and hence by Lemma~\ref{lem:g1}  we know that $D$ contains a clique $K_k$ with $k \ge \frac{n}{8a}$, such that the (total) degree of all its vertices is at least $2(n-1)-4a.$
	Let $n_0:=n_0(r)= 8a(40a^2+4a)+1$.
	Since $k \ge \frac{n}{8a} > 40a^2+4a= 32a^2+4a+8a^2 \ge 32a^2+4a+\frac{an}{k}$ when $n\ge n_0$, the result follows from Lemma~\ref{lem:g2}.
\end{proof}

\begin{prop}\label{proplem}
	Let $D=(V,A)$ be a digraph, $v \in V$ a vertex with outeccentricity $\ecc^+(v)=r' \ge r\ge 3$.
	Let $\P=vu_1u_2\ldots u_{r'}$ be a directed path in $D$ with $d(v,u_{r'})=r'$ such that for every vertex $u_i$, $r'-r+1 \le i \le r'$, there is a vertex $x_i$ with $d(v,x_i)<d(v,u_i)+d(u_i,x_i)$ such that $d(u_i,v)+d(v,x_i)\ge r$.
	Then \begin{equation}\label{Part1}
		\sum _{u \in V, u \not =v}\left(d(v,u)-1\right) + \sum _{1 \le i \le r'} \left(d(u_i,v)-1\right) \ge (r-1)^2 \end{equation} with equality if and only if the following conditions (up to labeling) are satisfied
	\begin{itemize}
		\item $r'=r$,
		\item there is a second directed path $v w_2w_3\ldots w_{r}$ with $d(v,w_{r})=r-1$ which is disjoint from the first directed path except from the vertex $v$,
		\item $d(u_i,v)=1$ for every $1 \le i \le r$, and
		\item  $d(v,u)=1$ for all vertices $u$ not on those two directed paths.	
	\end{itemize}
\end{prop}

\begin{proof}
	We will conclude by doing calculations over possible non-realizable graphs, i.e. we only take into account some important distances, the $d(u_i,v)$ and $d(v,x_i)$ and some substructure of shortest paths towards possible $x_i.$
	By doing this, it is easier to perform reductions such that we are able to focus on important substructures of the graph. As at the end equality can be obtained, there is no problem having intermediate steps where the sequence of distances is not realizable. 	
	
	For any $r'-r+1 \le i \le r'$, we let $x_i$ be a vertex on the path $\P$ if possible (that is when $d(u_i,v)+i-1 \ge r$).
	If this is not possible, we take a shortest (directed) path $\Q$ from $v$ to $x_i$ such that $\P \cap \Q$ is the directed path $vu_1 \ldots u_j$ with $j$ being minimal among all such possible choices for $\Q$ and we say $x_i$ is of type $j$.
	For every $j$, if there are multiple $x_i$ that have type $j$, we can associate the vertex $x_i$ such that $d(u_j,x_i)$ is maximal to all of them as we will only keep track of the distances $d(u_i,v)$ and do not look to possible shorter paths from a certain $u_i$ to $x_i$. 
	The shortest path from $u_j$ to the $x_i$ with $d(u_j,x_i)$ maximal will be denoted $Q_j.$

	Remark that $Q_{j_1}$ and $Q_{j_2}$ are disjoint for $j_1 <j_2$, since if there was a vertex in common, one could find a shortest path from $v$ of type $j_1$ in both cases.
	
	We will prove for $V^*$, being the set of all vertices in $\P$ and all paths $Q_j$, that
	\begin{equation}\label{eq:partialsum}
	\sum _{u \in V^*, u \not =v}\left(d(v,u)-1\right) + \sum _{1 \le i \le r'} \left(d(u_i,v)-1\right) \ge (r-1)^2.
	\end{equation}
	
	First we notice that there will be no $x_i$ of type $j$ for any $j \ge 2$ in an optimal possible graph. 
	Let the directed path $Q_j$ be $u_j z_1\ldots z_m x_i.$
	Note that if $x_k=x_i$ then $j+1 \le k \le j+d(u_j,x_i)$.
	When $k \le j$, then $d(v,x_i)=d(v,u_k)+d(u_k,x_i)$ and so one of the conditions for the choice of $x_k$ are not met.
	When $k > j+d(u_j,x_i)$ then it was possible to choose the vertex $x_k=u_{k-1}$ on $\P$ as $d(v,u_{k-1})<d(v,u_k)+d(u_k,u_{k-1})$ and $d(v,u_{k-1})=k-1\ge  j+d(u_j,x_i)=d(v,x_i)$ implying $d(u_k,v)+d(v,u_{k-1})\ge r$.
	
	If one removes $x_i$ of the digraph, takes $z_m$ as the new $x_i$ and increase all values $d(u_k,v)$ by one if $j+1 \le k \le j+d(u_j,x_i)$, then the conditions are still satisfied and the lefthandside of Equation~\eqref{eq:partialsum} decreased with at least $(j+d(u_j,x_i)-1)-d(u_j,x_i)=j-1.$
	Hence this was not optimal.

	Next, we want to prove there is no $Q_1$ as well in an optimal configuration. So suppose we have some $x_i$ of type $1$ and none of type at least $2$, so we only have the possible paths $Q_0$ and $Q_1.$
	If the path $Q_0$ its length $\ell_0$ is more than the length $\ell_1$ of $Q_1$, we can delete $Q_1.$
	In the other case we can delete $Q_0$ and $Q_1$ and add the path $Q_0=vw_2w_3 \ldots w_{\ell_1+1}x$ of length $\ell_1+1$ instead, as then the vertex $x$ can do the job of the $x_i$.
	If $d(u_1,v)>1$, we can decrease this distance with $1$ without destroying one of the necessary properties.
	If $d(u_1,v)=1$, then $r'>r$, or $r'=r$ which implies $Q_0$ had length at least $r-1$ and so some positive terms in the LHS of Equation~\eqref{eq:partialsum} disappeared.
	
	So having no $x_i$ of type $j \ge 1$, we know that there is some directed path $Q_0=vw_2w_3 \ldots w_s$ for some $2 \le s \le r'$ which is disjoint from $\P$ (except from $v$), or we could take $x_i=u_{i-1}$ in all cases (i.e. also $Q_0$ had length $0$).\\
	In the second case, we have that $d(u_i,v)+(i-1)=d(u_i,v)+d(v,u_{i-1}) \ge r$ and hence we can compute that the LHS of Equation~\eqref{eq:partialsum} is at least 
	\begin{align*}
	\sum_{1 \le i \le r'} \left(d(v,u_i)-1\right)+\sum _{r'-r+1 \le i \le r} \left(d(u_i,v)-1\right)&\ge \sum_{1 \le i \le r'} (i-1) +\sum _{r'-r+1 \le i \le r} (r-(i-1)-1)  \\
	&=\frac{r'(r'-1)}{2}+\frac{(2r-r')(2r-r'-1)}{2}\\
	&=\frac{(r'^2+(2r-r')^2)}{2}-r\\
	& \ge r^2-r>(r-1)^2.
	\end{align*}
	Here we used the inequality between the quadratic and arithmetic mean (QM-AM) at the end.
	In the first case, we have that $d(u_i,v)+s-1 \ge r$ when $r'-r+1 \le i \le s-1$ where we can assume $2 \le s \le r$, as $s\ge r+1$ implies the result trivially.
	When $s \le i \le r$, we have $d(u_i,v)+d(v,u_{i-1}) \ge r$.
	So we get that the LHS of Equation~\eqref{eq:partialsum} is at least 
	\begin{align*}
	&\sum_{i=1}^{r'} \left(d(v,u_i)-1 \right) + \sum_{i=2}^{s} \left(d(v,w_i)-1 \right) + \sum_{i=r'-r+1}^{s-1} \left(d(u_i,v)-1 \right) + \sum_{i=s}^{r} \left(d(u_i,v)-1 \right)\\
	&\ge \sum_{i=1}^{r'} (i-1) + \sum_{i=2}^{s} (i-2) + \sum_{i=r'-r+1}^{s-1} (r-s) + \sum_{i=s}^{r} (r-i)\\
	&= 
	\frac{r'(r'-1)}{2}+\frac{(s-1)(s-2)}{2}+(s+r-r'-1)(r-s)+\frac{(r-s)(r-s+1)}{2}.	\end{align*}
	This expression is strictly increasing for $r' \ge r$, so the minimum is attained when $r'=r$. So the expression reduces to $r^2-r-s+1$ which is minimal for $s=r$ and gives exactly $(r-1)^2.$ 
	Since equality cannot occur for $r'>r$, we conclude no other extremal cases could be left.
	In case of equality in Inequality~\eqref{Part1}, we also need $d(v,u)=1$ for vertices $u$ not considered (not on the paths), from which the characterization of the equality constraints is clear as well.	
\end{proof}

\begin{lem}\label{D2rr1core}
	Let $D=(V,A)$ be a digraph with outradius $r \ge 3$ containing directed paths 
	$w_1u_1u_2\ldots u_{r}$ and $w_1 w_2w_3\ldots w_{r}$ which are vertex-disjoint up to $w_1$ with $d(w_1,u_r)=r$ and $d(w_1,w_r)=r-1$.
	Let $V_1=\{w_r,w_{r-1} \ldots, w_2,w_1,u_1,u_2,\ldots, u_{r}\}$.
	Assume $d(u,w_1)=1$ for all $u \in V$ and $d(w_1,u)=1$ for all $u \in V \backslash V_1$.
	Then 
	 \begin{equation}\label{Part2}
		\sum _{x,y \in V_1} d(x,y) \ge W(D_{2r,r,1}) \end{equation} with equality if and only if $D[V_1]$ is isomorphic to $D_{2r,r,1}$. For any $y \in V \backslash V_1$ we have
		\begin{equation}\label{Part3}
		\sum_{x \in V_1} \left(d(x,y)+d(y,x)-2\right) \ge (r-1)^2 \end{equation}
	where equality occurs if and only if $D[v_1 \cup\{y\}]$ is isomorphic to $D_{2r+1,r,1}.$
\end{lem}

\begin{proof}[Proof of (\ref{Part2})]
	First note that for every $u \in V \backslash V_1$ and $v \in V$, we have $d(v,u) \le d(v,w_1)+d(w_1,u)\le 2$ and so $\ecc^+(v)\ge r$ implies there is a $x \in V_1$ with $d(v,x)\ge r$.
	
	By the given conditions, we know that $d(u_i,u_j)=d(w_i,w_j)=j-i$ when $1 \le i \le j \le r$ and $d(w_i, w_j),d(u_i,u_j) \ge 1$ if $j < i$.\\
	We also have $d(u_i,w_j)=j$ for all $1 \le i \le r-1$, $j$ being an upper bound since $d(u_i,w_1)=1$ and $d(w_1,w_j)=j-1$, while $d(u_i,w_j)<j$ for some $j$ would imply $d(u_i,w_r)<r$ and hence $\ecc^+(u_i)<r$, which would be a contradiction.
	
	Next, we see $d(w_i,u_j) \ge j$ for all $1 \le j \le r-1$.
	When $j=1$, this is by definition of distance and when $i=1$ this is a consequence of  the triangle-inequality $r=d(w_1,u_r)\le d(w_1,u_j)+d(u_j,u_r)=r-j +d(w_1,u_j). $
	When  $d(w_i,u_j) < j$ for some $1<j<r$ and $1<i$, we would get $\ecc^+(w_i)<r$, since $d(w_i,w_k)<r$ for all $k$, $d(w_i,u_k)\le 1+k<r$ when $k<j$ and $d(w_i,u_k)\le (j-1)+(k-j)=k-1 \le r-1$ when $k \ge j.$
	
	Inequality~\eqref{Part2} would is true if every term $d(x,y)$ is lowerbounded by the distance of the corresponding vertices in $D_{2r,r,1}$.
	The only distances for which we did not conclude that lowerbound already are of the form $d(u_r,w_j)$ for some $j$ or $d(w_i,u_r)$ for some $i$.
	The condition $d(u_r,w_j)<j$ for some $j>1$ would imply $d(u_r,w_j)=j-1$ since otherwise $\ecc^+(u_{r-1})<r$ and hence $d(u_r,u_{r-1})=r$ as $\ecc^+(u_r)=r$ and so $d(u_r,u_{i})\ge 1+i$ for all $1 \le i \le r-1.$
	Note this already implies that if $d(w_i,u_r)=r$ for all $i$, Inequality~\eqref{Part2} is strictly satisfied, since at most $r-1$ terms are one smaller than the corresponding term in $D_{2r,r,1}$, no other terms are smaller and the terms $d(u_r,u_{r-1})=r$ and $d(u_r,u_{1})=2$ are respectively $r-1$ and $1$ larger than the corresponding terms in $D_{2r,r,1}$.
	
	So we can focus on the case $d(w_i,u_r)<r$ for some $i$.
	Since $d(w_1,u_r)=r$, $d(w_1,w_i)=i-1$ and $d(w_i,u_{r-1})\ge r-1$, we see that $d(w_i,u_r)<r$ is only possible if there is an arc from $w_r$ to $u_r$ because there need to be an arc ending in $u_r$ different from $\vc{u_{r-1}u_r}$ and the other possibilities would lead to $d(w_1,u_r)<r.$
	So assume $d(w_r,u_r)=1$ and remark we now need $d(w_i, u_{r-1})=r$ for every $1<i \le r$,
	so $d(w_i,u_j) \ge j+1$ for all $1 \le j \le r-1$ and  $1<i \le r$.
	So at least $(r-1)^2$ distances are at least $1$ larger than the corresponding distances in $D_{2r,r,1}$, while the edge between $w_r$ and $u_r$ made we won only $\frac{r(r-1)}{2}$ with the terms corresponding to $d(w_i,u_r)$ for $2 \le i \le r$ and at most $r-1$ from the possibility that $d(u_r,w_j)=j-1$. But in that latter case, due to $d(u_r,u_{r-1})$ being necessarily equal to $r$, there is a term increasing with $r-1$ again. Since $\frac{ r(r-1)}2 < (r-1)^2$, Inequality~\eqref{Part3} is strict.

	We conclude that Inequality~\eqref{Part2} is always true and equality is possible if and only if $D[V_1]$ is isomorphic to $D_{2r,r,1}$.
\end{proof}

\begin{proof}[Proof of (\ref{Part3})]
	Note that $d(y,w_r) \ge r-2$ and $d(y,u_r)\ge r-1$ due to the conditions on $w_1$.
	As $\ecc^+(y)\ge r$, there is at least one vertex at distance $\ge r$, the only possible vertices for this are in $\{w_r,u_{r-1},u_r\}$.
	So we consider the three cases.\\
	If $d(y,w_r)=r$, combining this with $d(y,u_r) \ge r-1$ we already have \begin{align*}
	\sum_{x \in V_1} \left(d(y,x)-1\right) &\ge \sum_{i=2}^r (d(y,w_i)-1)+\sum_{i=2}^r (d(y,u_i)-1)\\
	&\ge \sum_{i=2}^r (i-1) + \sum_{i=2}^r (i-2)\\
	&=(r-1)^2.
	\end{align*} 
	The same holds if $d(y,u_r)\ge r$ or $d(y,u_{r-1})=r$ and $d(y,w_r)\ge r-1$. In the case $d(y,u_{r-1})=r$ it is even strict since $d(y,u_r)\ge r-1>1.$
	If $d(y,u_r)=r$ and $d(y,w_r)=r-2$, this implies $d(u_i,y)\ge 2$ for every $1\le i \le r$ as otherwise $\ecc^+(u_i)<r$.
	So now we have
	$$\sum_{x \in V_1} \left(d(x,y)-1\right) \ge r \mbox{ and }
	\sum_{x \in V_1} \left(d(y,x)-1\right) \ge \frac{r(r-1)}{2}+\frac{(r-2)(r-3)}{2}$$
	from which the result follows again as the sum is at least $(r-1)^2+1.$
	
	In the remaining case, we have $d(y,u_{r-1})=r$, $d(y,u_r) = r-1$ and $d(y,w_r) =r-2$ and the conclusion is again analogous.
	
	To have equality in Equation~\eqref{Part3}, we need $d(x,y)=1$ for every $x \in V_1$.
	Since $ecc^+(u_r)\ge r$, we need $d(y,w_r)\ge r-1$. The only cases we have to consider where equality can be attained, have $\{d(y,w_r),d(y,u_r)\}=\{r-1,r\}.$
	This implies that the only arcs from $y$ to $V_1$ can end in $w_2,w_1,v_1,v_2$ and the outneighbourhood cannot contain both $w_2$ and $v_2$. In the equality cases, $N^+(y)=\{w_2,w_1,v_1\}$ or $N^+(y)=\{w_1,v_1,v_2\}$.
\end{proof}

\begin{proof}[Proof of Theorem~\ref{maindi}]
Let $a=a(r)=(r-1)^2$ and $n_1:= n_1(r)=n_0(r)+a(r)n_0(r).$ Let $n \ge n_1.$
Note that a digraph $D$ with order $n \ge n_1$ which would be a counterexample to Theorem~\ref{maindi} will satisfy $W(D)<W(D_{n,r,1})<2\binom{n}{2}+a n$ due to Equation~\ref{eq:2}.
By Lemma~\ref{lem2di} we know that there is a sequence of $n-n_0 \ge an_0$ vertices $v_{n_0+1}, \ldots, v_{n}$ which are consecutively added, starting from some digraph $H=D_{n_0}$ with order $n_0$ and radius $r$, such that distances and the radius do not change.
Note that a digraph of radius $r$ satisfies the conditions of Proposition~\ref{proplem}.
If no addition of any of the $an_0$ vertices gives equality in Proposition~\ref{proplem}, then
\begin{align*}
W(D) &\ge W(H)+\sum_{i=n_0+1}^{n} \left(2i-1+(r-1)^2\right) \\
&> 2\binom{n_0}{2} + 2\sum_{i'=n_0}^{n-1} i' +(n-n_0)+(n-n_0)(r-1)^2\\
&\ge  2\binom{n}{2}+an_0+ (n-n_0)a\\
&= 2\binom{n}{2}+an
\end{align*}
which is a contradiction.
So we have equality in some step adding $v_m$ in Proposition~\ref{proplem} and we get a digraph $D_m$ at that step. Knowing the conditions of equality of Proposition~\ref{proplem} and that $D_m$ has outradius $r$, we may apply Lemma~\ref{D2rr1core} to conclude that $D_m$ should be of the form $D_{m,r,s}$ for some $s$ and so does the digraph at the final step.
So there was no digraph with smaller total distance and the extremal graphs are exactly of the form $D_{n,r,s}$.
\end{proof}

\section{Minimum for digraphs given order and radius}\label{subsec:Min_givenrad}

For small $r$, we easily can determine the exact minimum Wiener index of digraphs with given order and radius $r$. Note that $r$ can be an integer or a half-integer, i.e. $1 \le r$ with $r \in \frac12 \mathbb Z.$
For $r=2$, one can conclude from Theorem $3$ in~\cite{P84} that the digraph has diameter $2$ as any digraph with diameter at least $3$ has total distance at least equal to $n^2$ and the digraphs attaining equality have radius $\frac 32.$ 

\begin{prop}
	The minimum Wiener index among all digraphs $D$ with radius $r$ and order $n$ is at least
	$$
	\begin{cases}
	2\binom{n}{2} \hfill \mbox{ if } r=1,\\
	2\binom{n}{2}+\lceil \frac n2 \rceil \hfill \mbox{ if }r=\frac32 \mbox{ or}\\
	n^2 \hfill \mbox{ if }r=2 .\\
	\end{cases}
	$$
	Equality holds if and only if $D=K_n$, $D^c$ is the union of $\lceil \frac n2 \rceil$ directed edges which are spanning or $D^c$ is the union of some vertex disjoint directed cycles (possible of length $2$) which span all vertices.
\end{prop}

When $r \ge \frac 52$, one could guess that the extremal digraphs are blow-ups of simple structures closely related to e.g. a directed cycle $C_{\lceil r+1\rceil}$. 
It turns out that this is not the case.
For $r= \frac 52$, Figure~\ref{fig:digraph_rad2.5} presents at the left the structure of the extremal digraphs for $n \in \{4,5\}$ when taking $s=n-3.$ Nevertheless, for $n=6$ the structure is different as the structure at left has a larger total distance for $s=3$ than the digraph at the right~\footnote{See \url{https://github.com/StijnCambie/ChenWuAn}, document Digraph\_CWA.}. This implies that the characterization for $n \ge 2r$ cannot be as simple as it was in Conjecture~\ref{conjchen} and Conjecture~\ref{conjchendigraph}.

\begin{figure}[h]
	\centering
	\begin{tikzpicture}
	\definecolor{cv0}{rgb}{0.0,0.0,0.0}
	\definecolor{c}{rgb}{1.0,1.0,1.0}

	\Vertex[L=\hbox{$sK_1$},x=-1cm,y=0cm]{v0}
	\Vertex[L=\hbox{$v_3$},x=3.0cm,y=0]{v1}
	\Vertex[L=\hbox{$v_1$},x=2cm,y=2cm]{v2}
	\Vertex[L=\hbox{$v_2$},x=0cm,y=2cm]{v3}

	\Edge[lw=0.1cm,style={post, right}](v0)(v1)
	\Edge[lw=0.1cm,style={post, right}](v1)(v0)
	\Edge[lw=0.1cm,style={post, right}](v1)(v2)
	\Edge[lw=0.1cm,style={post, right}](v2)(v3)
	\Edge[lw=0.1cm,style={post, right}](v3)(v2)
	\Edge[lw=0.1cm,style={post, right}](v3)(v0)
	\end{tikzpicture}
	\quad
	\begin{tikzpicture}
	\definecolor{cv0}{rgb}{0.0,0.0,0.0}
	\definecolor{cfv0}{rgb}{1.0,1.0,1.0}
	\definecolor{clv0}{rgb}{0.0,0.0,0.0}
	\definecolor{cv1}{rgb}{0.0,0.0,0.0}
	\definecolor{cfv1}{rgb}{1.0,1.0,1.0}
	\definecolor{clv1}{rgb}{0.0,0.0,0.0}
	\definecolor{cv2}{rgb}{0.0,0.0,0.0}
	\definecolor{cfv2}{rgb}{1.0,1.0,1.0}
	\definecolor{clv2}{rgb}{0.0,0.0,0.0}
	\definecolor{cv3}{rgb}{0.0,0.0,0.0}
	\definecolor{cfv3}{rgb}{1.0,1.0,1.0}
	\definecolor{clv3}{rgb}{0.0,0.0,0.0}
	\definecolor{cv4}{rgb}{0.0,0.0,0.0}
	\definecolor{cfv4}{rgb}{1.0,1.0,1.0}
	\definecolor{clv4}{rgb}{0.0,0.0,0.0}
	\definecolor{cv5}{rgb}{0.0,0.0,0.0}
	\definecolor{cfv5}{rgb}{1.0,1.0,1.0}
	\definecolor{clv5}{rgb}{0.0,0.0,0.0}
	\definecolor{cv0v1}{rgb}{0.0,0.0,0.0}
	\definecolor{cv0v3}{rgb}{0.0,0.0,0.0}
	\definecolor{cv1v0}{rgb}{0.0,0.0,0.0}
	\definecolor{cv1v2}{rgb}{0.0,0.0,0.0}
	\definecolor{cv2v0}{rgb}{0.0,0.0,0.0}
	\definecolor{cv2v1}{rgb}{0.0,0.0,0.0}
	\definecolor{cv2v3}{rgb}{0.0,0.0,0.0}
	\definecolor{cv2v5}{rgb}{0.0,0.0,0.0}
	\definecolor{cv3v0}{rgb}{0.0,0.0,0.0}
	\definecolor{cv3v4}{rgb}{0.0,0.0,0.0}
	\definecolor{cv4v0}{rgb}{0.0,0.0,0.0}
	\definecolor{cv4v1}{rgb}{0.0,0.0,0.0}
	\definecolor{cv4v3}{rgb}{0.0,0.0,0.0}
	\definecolor{cv4v5}{rgb}{0.0,0.0,0.0}
	\definecolor{cv5v1}{rgb}{0.0,0.0,0.0}
	\definecolor{cv5v2}{rgb}{0.0,0.0,0.0}
	\definecolor{cv5v3}{rgb}{0.0,0.0,0.0}
	\definecolor{cv5v4}{rgb}{0.0,0.0,0.0}
\Vertex[style={minimum size=1.0cm,draw=cv0,fill=cfv0,text=clv0,shape=circle},LabelOut=false,L=\hbox{$0$},x=0cm,y=0.0cm]{v0}
\Vertex[style={minimum size=1.0cm,draw=cv1,fill=cfv1,text=clv1,shape=circle},LabelOut=false,L=\hbox{$1$},x=0cm,y=1.75cm]{v1}
\Vertex[style={minimum size=1.0cm,draw=cv2,fill=cfv2,text=clv2,shape=circle},LabelOut=false,L=\hbox{$2$},x=1.25cm,y=3cm]{v2}
\Vertex[style={minimum size=1.0cm,draw=cv3,fill=cfv3,text=clv3,shape=circle},LabelOut=false,L=\hbox{$3$},x=3cm,y=1.25cm]{v3}
\Vertex[style={minimum size=1.0cm,draw=cv4,fill=cfv4,text=clv4,shape=circle},LabelOut=false,L=\hbox{$4$},x=1.75cm,y=0cm]{v4}
\Vertex[style={minimum size=1.0cm,draw=cv5,fill=cfv5,text=clv5,shape=circle},LabelOut=false,L=\hbox{$5$},x=3cm,y=3cm]{v5}
	\Edge[lw=0.1cm,style={post, color=cv0v1,},](v0)(v1)
\Edge[lw=0.1cm,style={post, color=cv0v3,},](v0)(v3)
\Edge[lw=0.1cm,style={post, color=cv1v0,},](v1)(v0)
\Edge[lw=0.1cm,style={post, color=cv1v2,},](v1)(v2)
\Edge[lw=0.1cm,style={post, color=cv2v0,},](v2)(v0)
\Edge[lw=0.1cm,style={post, color=cv2v1,},](v2)(v1)
\Edge[lw=0.1cm,style={post, color=cv2v3,},](v2)(v3)
\Edge[lw=0.1cm,style={post, color=cv2v5,},](v2)(v5)
\Edge[lw=0.1cm,style={post, color=cv3v0,},](v3)(v0)
\Edge[lw=0.1cm,style={post, color=cv3v4,},](v3)(v4)
\Edge[lw=0.1cm,style={post, color=cv4v0,},](v4)(v0)
\Edge[lw=0.1cm,style={post, color=cv4v1,},](v4)(v1)
\Edge[lw=0.1cm,style={post, color=cv4v3,},](v4)(v3)
\Edge[lw=0.1cm,style={post, color=cv4v5,},](v4)(v5)
\Edge[lw=0.1cm,style={post, color=cv5v1,},](v5)(v1)
\Edge[lw=0.1cm,style={post, color=cv5v2,},](v5)(v2)
\Edge[lw=0.1cm,style={post, color=cv5v3,},](v5)(v3)
\Edge[lw=0.1cm,style={post, color=cv5v4,},](v5)(v4)
	\end{tikzpicture}

	\caption{Digraph with $\rad=2.5$ and minimum Wiener index for $n \le 6$}
	\label{fig:digraph_rad2.5}
\end{figure}

The strategies used before, are also harder in this case.
The analog of Proposition~\ref{proplem} for radius does not give a unique nor clear configuration or substructure for example. Some examples for this are presented in Figure~\ref{fig:digraph_rad} and Figure~\ref{fig:digraph_rad3.5} in case the radius is $3$ or $3.5$.
We note that taking a blow-up in vertices $2$ or $4$ of these configurations gives the minimum total distance up to a constant, as a corollary of the ideas used before.
We give a sketch of that partial result in this section for $r \ge \frac 72.$
 
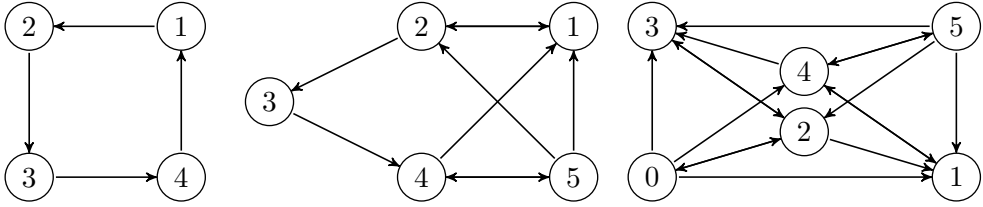
\begin{figure}[h]
	\centering

	\begin{tikzpicture}
	\definecolor{cv0}{rgb}{0.0,0.0,0.0}
	\definecolor{c}{rgb}{1.0,1.0,1.0}

	\Vertex[L=\hbox{$1$},x=4cm,y=1cm]{v0}
	\Vertex[L=\hbox{$2$},x=2.0cm,y=1]{v1}
	\Vertex[L=\hbox{$3$},x=2cm,y=-1cm]{v3}
	\Vertex[L=\hbox{$4$},x=4cm,y=-1]{v4}
	
	\Edge[lw=0.1cm,style={post, right}](v0)(v1)
	\Edge[lw=0.1cm,style={post, right}](v1)(v3)
	\Edge[lw=0.1cm,style={post, right}](v3)(v4)
	\Edge[lw=0.1cm,style={post, right}](v4)(v0)
	\end{tikzpicture}	
	\quad
	\begin{tikzpicture}
	\definecolor{cv0}{rgb}{0.0,0.0,0.0}
	\definecolor{c}{rgb}{1.0,1.0,1.0}

	\Vertex[L=\hbox{$1$},x=4cm,y=1cm]{v0}
	\Vertex[L=\hbox{$2$},x=2.0cm,y=1]{v1}
	\Vertex[L=\hbox{$3$},x=0cm,y=0cm]{v2}
	\Vertex[L=\hbox{$4$},x=2cm,y=-1cm]{v3}
	\Vertex[L=\hbox{$5$},x=4cm,y=-1]{v4}
	
	\Edge[lw=0.1cm,style={post, right}](v0)(v1)
	\Edge[lw=0.1cm,style={post, right}](v1)(v0)
	\Edge[lw=0.1cm,style={post, right}](v1)(v2)
	\Edge[lw=0.1cm,style={post, right}](v4)(v1)
	\Edge[lw=0.1cm,style={post, right}](v4)(v3)
	\Edge[lw=0.1cm,style={post, right}](v2)(v3)
	\Edge[lw=0.1cm,style={post, right}](v3)(v4)
	\Edge[lw=0.1cm,style={post, right}](v3)(v0)
	
	\Edge[lw=0.1cm,style={post, right}](v4)(v0)
	\end{tikzpicture}\quad
	\begin{tikzpicture}
	\definecolor{cv0}{rgb}{0.0,0.0,0.0}
	\definecolor{cfv0}{rgb}{1.0,1.0,1.0}
	\definecolor{clv0}{rgb}{0.0,0.0,0.0}
	\definecolor{cv1}{rgb}{0.0,0.0,0.0}
	\definecolor{cfv1}{rgb}{1.0,1.0,1.0}
	\definecolor{clv1}{rgb}{0.0,0.0,0.0}
	\definecolor{cv2}{rgb}{0.0,0.0,0.0}
	\definecolor{cfv2}{rgb}{1.0,1.0,1.0}
	\definecolor{clv2}{rgb}{0.0,0.0,0.0}
	\definecolor{cv3}{rgb}{0.0,0.0,0.0}
	\definecolor{cfv3}{rgb}{1.0,1.0,1.0}
	\definecolor{clv3}{rgb}{0.0,0.0,0.0}
	\definecolor{cv4}{rgb}{0.0,0.0,0.0}
	\definecolor{cfv4}{rgb}{1.0,1.0,1.0}
	\definecolor{clv4}{rgb}{0.0,0.0,0.0}
	\definecolor{cv5}{rgb}{0.0,0.0,0.0}
	\definecolor{cfv5}{rgb}{1.0,1.0,1.0}
	\definecolor{clv5}{rgb}{0.0,0.0,0.0}
	\definecolor{cv0v1}{rgb}{0.0,0.0,0.0}
	\definecolor{cv0v2}{rgb}{0.0,0.0,0.0}
	\definecolor{cv0v3}{rgb}{0.0,0.0,0.0}
	\definecolor{cv0v4}{rgb}{0.0,0.0,0.0}
	\definecolor{cv1v3}{rgb}{0.0,0.0,0.0}
	\definecolor{cv2v0}{rgb}{0.0,0.0,0.0}
	\definecolor{cv2v1}{rgb}{0.0,0.0,0.0}
	\definecolor{cv2v4}{rgb}{0.0,0.0,0.0}
	\definecolor{cv3v1}{rgb}{0.0,0.0,0.0}
	\definecolor{cv3v4}{rgb}{0.0,0.0,0.0}
	\definecolor{cv3v5}{rgb}{0.0,0.0,0.0}
	\definecolor{cv4v2}{rgb}{0.0,0.0,0.0}
	\definecolor{cv5v1}{rgb}{0.0,0.0,0.0}
	\definecolor{cv5v2}{rgb}{0.0,0.0,0.0}
	\definecolor{cv5v3}{rgb}{0.0,0.0,0.0}
	\definecolor{cv5v4}{rgb}{0.0,0.0,0.0}
	\Vertex[style={minimum size=1.0cm,draw=cv0,fill=cfv0,text=clv0,shape=circle},LabelOut=false,L=\hbox{$0$},x=0cm,y=0.0cm]{v0}
	\Vertex[style={minimum size=1.0cm,draw=cv1,fill=cfv1,text=clv1,shape=circle},LabelOut=false,L=\hbox{$1$},x=4.0cm,y=0cm]{v1}
	\Vertex[style={minimum size=1.0cm,draw=cv2,fill=cfv2,text=clv2,shape=circle},LabelOut=false,L=\hbox{$2$},x=2cm,y=0.6cm]{v2}
	\Vertex[style={minimum size=1.0cm,draw=cv3,fill=cfv3,text=clv3,shape=circle},LabelOut=false,L=\hbox{$4$},x=2cm,y=1.4cm]{v3}
	\Vertex[style={minimum size=1.0cm,draw=cv4,fill=cfv4,text=clv4,shape=circle},LabelOut=false,L=\hbox{$3$},x=0.0cm,y=2cm]{v4}
	\Vertex[style={minimum size=1.0cm,draw=cv5,fill=cfv5,text=clv5,shape=circle},LabelOut=false,L=\hbox{$5$},x=4cm,y=2cm]{v5}
	\Edge[lw=0.1cm,style={post, color=cv0v1,},](v0)(v1)
	\Edge[lw=0.1cm,style={post, color=cv0v2,},](v0)(v2)
	\Edge[lw=0.1cm,style={post, color=cv0v3,},](v0)(v3)
	\Edge[lw=0.1cm,style={post, color=cv0v4,},](v0)(v4)
	\Edge[lw=0.1cm,style={post, color=cv1v3,},](v1)(v3)
	\Edge[lw=0.1cm,style={post, color=cv2v0,},](v2)(v0)
	\Edge[lw=0.1cm,style={post, color=cv2v1,},](v2)(v1)
	\Edge[lw=0.1cm,style={post, color=cv2v4,},](v2)(v4)
	\Edge[lw=0.1cm,style={post, color=cv3v1,},](v3)(v1)
	\Edge[lw=0.1cm,style={post, color=cv3v4,},](v3)(v4)
	\Edge[lw=0.1cm,style={post, color=cv3v5,},](v3)(v5)
	\Edge[lw=0.1cm,style={post, color=cv4v2,},](v4)(v2)
	\Edge[lw=0.1cm,style={post, color=cv5v1,},](v5)(v1)
	\Edge[lw=0.1cm,style={post, color=cv5v2,},](v5)(v2)
	\Edge[lw=0.1cm,style={post, color=cv5v3,},](v5)(v3)
	\Edge[lw=0.1cm,style={post, color=cv5v4,},](v5)(v4)
	\end{tikzpicture}
	
	\centering

	\caption{Digraphs with $\rad=3$ and smallest Wiener index for $n\le 6$}
	\label{fig:digraph_rad}
\end{figure}

\begin{figure}[h]
	\centering

	\begin{tikzpicture}
	\definecolor{cv0}{rgb}{0.0,0.0,0.0}
	\definecolor{c}{rgb}{1.0,1.0,1.0}

\Vertex[L=\hbox{$1$},x=4cm,y=1cm]{v0}
\Vertex[L=\hbox{$2$},x=2.0cm,y=1]{v1}
\Vertex[L=\hbox{$3$},x=0cm,y=0cm]{v2}
\Vertex[L=\hbox{$4$},x=2cm,y=-1cm]{v3}
\Vertex[L=\hbox{$5$},x=4cm,y=-1]{v4}

\Edge[lw=0.1cm,style={post, right}](v0)(v1)
\Edge[lw=0.1cm,style={post, right}](v1)(v0)
\Edge[lw=0.1cm,style={post, right}](v1)(v2)
\Edge[lw=0.1cm,style={post, right}](v4)(v3)
\Edge[lw=0.1cm,style={post, right}](v2)(v3)
\Edge[lw=0.1cm,style={post, right}](v3)(v4)
\Edge[lw=0.1cm,style={post, right}](v4)(v0)

\Edge[lw=0.1cm,style={post, right}](v4)(v0)
\end{tikzpicture}	
\quad
	\begin{tikzpicture}
	\definecolor{cv0}{rgb}{0.0,0.0,0.0}
	\definecolor{c}{rgb}{1.0,1.0,1.0}

	\Vertex[L=\hbox{$2$},x=2.5cm,y=0.8cm]{v0}
	\Vertex[L=\hbox{$1$},x=4cm,y=0]{v1}
	\Vertex[L=\hbox{$3$},x=0cm,y=0cm]{v2}
	\Vertex[L=\hbox{$4$},x=0cm,y=2.5cm]{v3}
	\Vertex[L=\hbox{$5$},x=4cm,y=2.5]{v4}
	\Vertex[L=\hbox{$6$},x=2cm,y=1.9]{v5}
	
	\Edge[lw=0.1cm,style={post, right}](v0)(v1)
	\Edge[lw=0.1cm,style={post, right}](v1)(v0)
	\Edge[lw=0.1cm,style={post, right}](v0)(v2)
	\Edge[lw=0.1cm,style={post, right}](v2)(v0)
	\Edge[lw=0.1cm,style={post, right}](v2)(v1)
	\Edge[lw=0.1cm,style={post, right}](v2)(v3)
	\Edge[lw=0.1cm,style={post, right}](v2)(v5)
	\Edge[lw=0.1cm,style={post, right}](v4)(v1)
	\Edge[lw=0.1cm,style={post, right}](v3)(v4)
	\Edge[lw=0.1cm,style={post, right}](v4)(v3)
	\Edge[lw=0.1cm,style={post, right}](v3)(v5)
	\Edge[lw=0.1cm,style={post, right}](v5)(v3)	
	\Edge[lw=0.1cm,style={post, right}](v4)(v5)
	
	\Edge[lw=0.1cm,style={post, right}](v0)(v5)
	
	\end{tikzpicture}
	\caption{Digraphs with $\rad=3.5$ and minimum Wiener index for $n \in \{5,6\}$}
	\label{fig:digraph_rad3.5}
\end{figure}
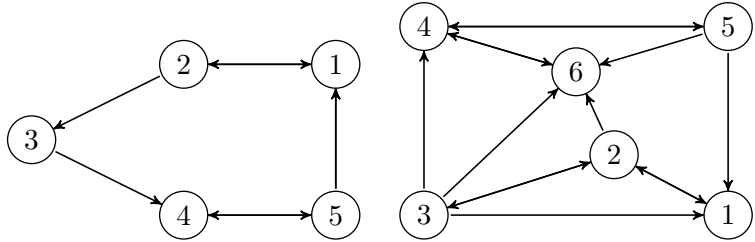

We now prove the asymptotic result roughly by showing the corresponding version of what has been done for the out-radius. 
Lemma~\ref{lem:g1} does not use the notion of out-radius, but we note that the numbers can be chosen a bit different and can be straightforward generalized to Lemma~\ref{lem:g1'}. Note that the idea behind is exactly the same as in the proof of Markov's inequality if you state it as \hyphenquote{UKenglish}{the probability that a vertex has at least $b$ times the average degree, is at most $\frac 1b$}, where one has to apply this to the complement graph. 

\begin{lem}\label{lem:g1'}
	Let $D$ be a digraph with average total degree at least being equal to $2(n-1)-2t$. Then at most $\frac 1b$ of the vertices have a total degree which is at most $2(n-1)-2b t.$ Also at most $\frac 1b$ of the vertices have in-degree at most $(n-1)-bt$ and the same holds for out-degree being at most $(n-1)-bt$.	
	Furthermore $D$ contains a bidirected clique of size at least $\frac{n}{8t}$ all of whose vertices have total degree more than $2(n-1)-4t.$ 
\end{lem}

The analog of Lemma~\ref{lem:g2} and Lemma~\ref{lem2di} needs some more care and we only prove $r \ge \frac 72$ to avoid even more details. 

	\begin{lem}\label{lem2di'}

		Let $r \ge \frac 72$. Let $a= \lfloor (r-0.5)^2 \rfloor.$
		There is a value $n_0(r)$ such that for any $n \ge n_0$ and any digraph $D$ of order $n$ and radius $r$ with $W(D)<2\binom{n}{2}+a n$ for some positive integer $a$, there is a vertex $v \in D$ such that $D \backslash v$ has outradius $r$ and the distance between any $2$ vertices of $D \backslash v$ equals the distance between them in $D.$

\end{lem}

\begin{proof}
	Note that the size of the digraph is at least $n(n-1-a)$ again and by Lemma~\ref{lem:g1'} there exists a bidirected clique  $K_k$ with $k \ge \frac n {8a}$ all of whose vertices have total degree more than $2(n-1)-4a.$
	As was done in Lemma~\ref{lem:g2}, we know there exists a set $S$ of at most $32a^2+4a+1$ vertices of $K_k$ 
	such that for any $2$ vertices $x,y$ in $D \backslash K_k$ for which there exists a vertex $v \in K_k$ such that $\vc{xv}$ and $\vc{vy}$ are edges of $D$, there is an $s \in S$ with $\vc{xs}$ and $\vc{sy}$ being edges of $D$. Also for any vertex $v$ in $D \backslash K_k$ for which there is an arc towards or from a vertex $u$ in $K_k$, there is such an arc $\vc{vs}$ or $\vc{sv}$ for some $s \in S.$
	This implies that both the in- and outeccentricities of the remaining vertices in $D\backslash z$ for any $z \in K_k \backslash S$ cannot increase.
	As a consequence, there is at most one vertex $z^* \in K_k \backslash S$ such that the radius of $D \backslash z^*$ is larger than $r.$
	Next, we invest vertices $z \in K_k \backslash S$ such that the radius in $D\backslash z$ becomes smaller.
	Then there must be a vertex $w$ for which $\ecc^+$ or $\ecc^-$ strictly decreases.
	Similarly as in Lemma~\ref{lem:g2}, in case $d(w,z)$ or $d(z,w)$ equals at least $3$ and the distance from or towards the other vertices is smaller, we can see that there is only one choice for $z$ such that this is the case for a particular $w$ and there are most $\frac{an}{k}$ of them in any of the two directions.
	
	Now look to the number of vertices $z \in K_k \backslash S$ such that $d(w,z)=2$ for a choice of $w$ and $d(w,v)=1$ for all vertices $v$ different from $w$ and $z.$ Assume there are at least $b(r)a+1$ of them. 
	Then as a corollary of Lemma~\ref{lem:g1'} at least $\left( 1- \frac1b\right)n$ vertices are connected with at least one of these vertices $w$.
	Since at most $4a$ vertices do not have a directed edge towards the corresponding $z$, there are at least $\left( 1- \frac1b\right)n-4a$ vertices having out-eccentricity at most $2$.
	So all of them have in-eccentricity at least $2r-2$. Let $U$ be the set of these vertices.
	For every $u$, the vertices $v$ for which $d(v,u)=m$ for any $3 \le m \le 2r-2$ do not belong to $U,$ as $v \in U$ implies $d(v,V)\le 2.$
	This implies that 
	$$\sum_{u \in U, v \in V \backslash U} \left( d(v,u) -1 \right) \ge 
	\left( \left( 1- \frac1b\right)n-4a \right) \cdot \left( \sum_{i=3}^{2r-2} (i-1) \right).$$
	There are also at least $\frac n2$ pairs of vertices $x,y \in V$ such that $d(x,y)=2.$
	For $b$ and $n$ sufficiently large, this would lead to a contradiction with $W(D)< 2 \binom{n}2 + an$ since 
	$\frac 12 + \sum_{i=3}^{2r-2} (i-1)=(r-2)(2r-1) + \frac 12 > a$ when $r \ge \frac 72.$
	Similarly there are less than $b(r)a+1$ vertices $z \in K_k \backslash S$ such that $d(w,z)=2$ for a choice of $w$ and $d(w,v)=1$ for all vertices $v \not =z.$
	If $n_0$ (and $b$) is large enough to conclude $k \ge 32a^2+4a+1 + 2 \frac{an}{k} + 2(b(r)a+2)$, we know there exists a vertex satisfying the conditions of the lemma. 	
\end{proof}

The analog of Proposition~\ref{proplem} without characterization is immediate.

\begin{prop}\label{proplem'}
	Let $D=(V,A)$ be a digraph, with radius $r.$

	Then \begin{equation}\label{Part1'}
	\sum_{v \in V \backslash x} \left( d(x,v)-1+d(v,x)-1 \right) \ge \lfloor (r-0.5)^2 \rfloor. \end{equation} 
\end{prop}

\begin{proof}
	By definition of the radius, for every vertex $x \in V$, we have $d(x,V)+d(V,x)\ge 2r$. Let $a=d(x,V)$ and $b=d(V,x).$
	Then $$\sum_{v \in V \backslash x} \left( d(x,v)-1+d(v,x)-1 \right) \ge \frac{a(a-1)}2+\frac{b(b-1)}2 \ge \frac{ \lfloor r \rfloor^2 +  \lceil r \rceil^2 - 2r}{2}= \lfloor (r-0.5)^2 \rfloor. \qedhere$$
\end{proof}

\begin{thr}\label{min_rad}
	For $r \ge \frac 72$, the minimum Wiener index among all digraphs with radius $r$ and order $n$ is of the form $2\binom{n}2 + \lfloor (r-0.5)^2 \rfloor n - \Theta_r(1).$
\end{thr}

\begin{proof}
	First we note that digraphs of the desired form do exist.
	When $r \in \mathbb N$, we can take the blow-up of a vertex of a directed cycle $C_{r+1}.$
	When $r \in \frac{1}{2}+\mathbb N$, we can take a directed cycle $C_{r+1.5}$ with an additional directed edge in the opposite direction between two neighbours.
	Now take a blow-up of the startvertex of that additional directed edge.

	Let $a:=a(r)=\lfloor (r-0.5)^2 \rfloor$. Choose $n_0$ as in Lemma~\ref{lem2di'}.
	Then for $n \ge n_0$ and any extremal digraph $D$ or order $n$ and radius $r$ we know that there are $n-n_0$ vertices satisfying the conclusion of Lemma~\ref{lem2di'}. Let $H$ be the digraph induced by the other $n_0$ vertices.
	We now can compute the Wienerindex in reverse order again using Proposition~\ref{proplem'}. With $a=\lfloor (r-0.5)^2 \rfloor$, we have
	\begin{align*}
	W(D) &\ge W(H)+\sum_{i=n_0+1}^{n} \left(2i-2+ a \right) \\
	&> 2\binom{n_0}{2} + 2\sum_{i'=n_0}^{n-1} i' +(n-n_0)a\\
	&\ge  2\binom{n}{2}+an-  an_0.
	\end{align*}
	Note that $an_0 = \Theta(r)$ to conclude.
\end{proof}

\section{Conclusion}\label{conc}
	The question of determining the minimum total distance among all graphs or digraphs of order $n$ and (out)radius $r$ and characterizing the extremal (di)graphs has been solved for $n$ large enough compared with $r$. This partially solves the conjecture of Chen, Wu and An~\cite{Chen}.
	It might be very challenging to solve the question completely for every order, as there might be sporadic extremal graphs other than $Q_3$ which are also different from those conjectured in the initial conjecture. Also for the digraph version, there might be some counterexamples.
	Note that a digraph with outradius $r$ only needs $n \ge r+1$ instead of $n \ge 2r$. For $n=r+1$, the only digraph with outradius $r$ is the directed cycle $C_{r+1}.$
	For $r+2 \le n \le 2r-1$, the extremal graphs may be blow-ups of $C_{r+1}$, but this is not verified.
	
	The analog for the radius in the digraph case was considered as well.
	Nevertheless, it seems hard to get the exact lowerbound and extremal digraph(s) for all values of $r$ and $n$.
	
	\paragraph{Open access statement.} For the purpose of open access,
	a CC BY public copyright license is applied
	to any Author Accepted Manuscript (AAM)
	arising from this submission.

\bibliographystyle{abbrv}
\bibliography{MaxMuD}

\end{document}